\documentclass{amsart}
\usepackage[utf8]{inputenc}
\usepackage{graphicx}
\usepackage{array}
\usepackage{tabularx}
\usepackage{amssymb}
\usepackage{tikz}
\usepackage{comment}
\usepackage{tikz-cd}
\usetikzlibrary{graphs}
\usepackage{blindtext}
\usepackage{hyperref}\hypersetup{colorlinks,allcolors=black}
\usepackage{varwidth}
\usepackage{import}
\usepackage{float}
\usepackage{amsmath}
\usepackage{amsthm}
\usepackage{calc}
 
\usepackage{caption}
\usepackage{subcaption}
\usepackage{wrapfig}
\newcommand{\handle}[1]{$#1$\text{-handle}}
\newcommand{\handles}[1]{$#1$\text{-handles}}
\newcommand{\Kd}{Kirby diagram\;}
\newcommand{\Kds}{Kirby diagrams\;}
\newcommand{\Hd}{Heegaard diagram\;}
\newcommand{\Hds}{Heegaard diagrams\;}

\makeatletter
\newtheorem*{rep@algorithm}{\rep@title}
\newcommand{\newrepalgorithm}[2]{%
\newenvironment{rep#1}[1]{%
 \def\rep@title{#2 \ref{##1}}%
 \begin{rep@algorithm}}%
 {\end{rep@algorithm}}}
\makeatother
\makeatletter
\newtheorem*{rep@theorem}{\rep@title}
\newcommand{\newreptheorem}[2]{%
\newenvironment{rep#1}[1]{%
 \def\rep@title{#2 \ref{##1}}%
 \begin{rep@theorem}}%
 {\end{rep@theorem}}}
\makeatother
\theoremstyle{definition}
\newtheorem{definition}{Definition}[section]
\newtheorem{theorem}[definition]{Theorem}

\newtheorem{lemma}[definition]{Lemma}
\newtheorem{corollary}[definition]{Corollary}
\newtheorem{proposition}[definition]{Proposition}
\newtheorem{remark}[definition]{Remark}
\newtheorem{example}[definition]{Example}
\newtheorem{algorithm}[definition]{Algorithm}
\newreptheorem{theorem}{Theorem}
\newrepalgorithm{algorithm}{Algorithm}

\title{Kirby diagrams of 4-dimensional open books}
\author{Chun-Sheng Hsueh}
\address{Humboldt Universit\"at zu Berlin, Germany}
\email{chun-sheng.hsueh@hu-berlin.de}
\begin{document}

\begin{abstract}
   We provide an algorithm for constructing a \Kd of a $4$-dimensional open book given a \Hd of the page. As an application, we show that any open book with trivial monodromy is diffeomorphic to an open book constructed with a punctured handlebody as page and a composition of torus twists and sphere twists as monodromy.
\end{abstract}
\maketitle

\section{Introduction}\label{intro}
Obstructions to the existence of open books are found in all dimensions and are known to be complete in all dimensions except $4$~\cite{Quinn}. Kirby calculus is a successful approach to studying $4$-dimensional manifolds~\cite{gs}, which we would like to apply to study open books in dimension $4$. 

We introduce the notion of a \textit{half open book}, a generalization of a “lens thickening”~\cite{Baader}, whose \Kd is well known. Our Algorithm~\ref{algorithm3} for constructing a \Kd of an open book simply involves adding a framed link to a \Kd of a half open book, which corresponds to gluing two half open books together. This will quickly yield Theorem~\ref{not unique}, where we show that an open book constructed with an arbitrary page and a trivial monodromy can be alternatively constructed using a simple page, namely a punctured handlebody, and a simple but nontrivial monodromy composed of sphere twists and torus twists. A handlebody is a $3$-manifold with boundary obtained from $D^3$ by attaching \handles{1}, and by a punctured handlebody we mean a handlebody with open $3$-balls removed from its interior. Finally, we prove that the diffeomorphism type of the spin of a lens space $\operatorname{L}(p,q)$ is independent of $q$, a result due to~\cite{Meier,Pao,Plotnick}, using Kirby calculus. Section~\ref{examples} is dedicated to examples that effectively capture the main ideas.

We start by recalling the definitions necessary to state our main results. Given a compact oriented $3$-manifold $M$ with non-empty boundary called a \textit{page} and a self-diffeomorphism $\varphi\colon M\rightarrow M$ that restricts to the identity map in a neighborhood of $\partial M$ called the \textit{monodromy}, the \textit{open book} $\operatorname{Ob}(M,\varphi)$ is diffeomorphic to the closed $4$-manifold $M\times [0,1]/\sim_\varphi$, where the equivalence relation is given by 
$$(x,1)\sim_\varphi(\varphi(x),0)\text{ for all } x\in M,\text{ and}$$ $$(x,t)\sim_\varphi(x,t')\text{ for all } x\in \partial M, t,t'\in [0,1].$$ 

\begin{reptheorem}{not unique}
    Given any compact oriented $3$-manifold $M$ with non-empty boundary, there exists a monodromy $\varphi$ on a punctured handlebody $H$ such that $\operatorname{Ob}(H,\varphi)$ is diffeomorphic to $\operatorname{Ob}(M,\operatorname{id})$.
\end{reptheorem}

\begin{figure}[ht]
    \centering
    \includegraphics[width=\textwidth]{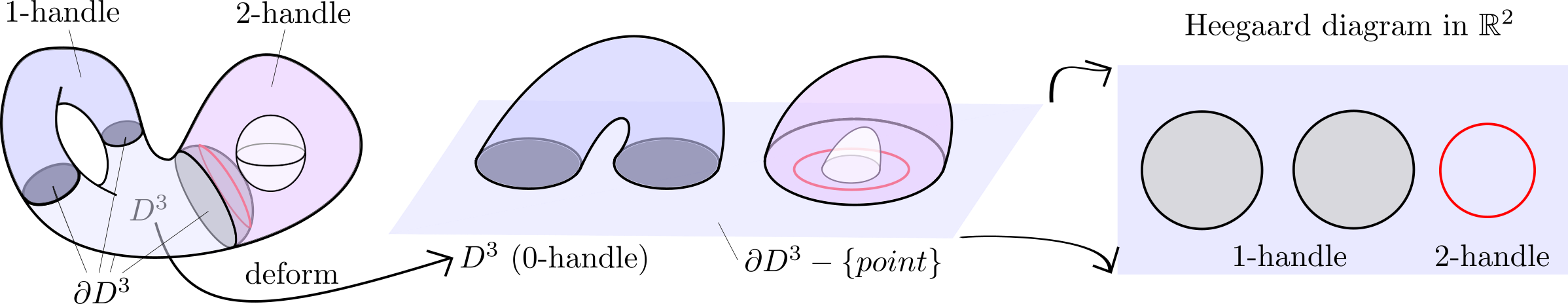}
    \caption{Handle decomposition to \Hd (of the punctured solid torus).}
    \label{fig:hd of punctured solid torus}
\end{figure}
A \textit{\Hd} is a way of presenting a $3$-manifold with non-empty boundary in $\mathbb{R}^2$. We think of $\mathbb{R}^2$ as the boundary of the \handle{0} $\partial D^3=S^2$ with a single point removed to which we attach 1- and \handles{2}. The attaching region of a \handle{1} is indicated with a pair of $D^2$'s and the attaching sphere of a \handle{2} is represented by a simple closed curve. Similarly, one can regard $\mathbb{R}^3$ as the boundary of the $4$-dimensional \handle{0} with a point removed. In the $4$-dimensional case, the attaching region of a \handle{1} is specified by a pair of $D^3$'s. In addition to marking the attaching sphere, the framing of a \handle{2} also needs to be specified. A parallel curve of the attaching sphere of a \handle{2} can be used to encode a trivialization of its tubular neighborhood in $\mathbb{R}^3$. As a consequence of~\cite{Laudenbach}, there is essentially a unique way to obtain a closed $4$-manifold by attaching 3- and \handles{4}. Hence such a diagram, called a \textit{Kirby diagram}, defines a closed $4$-manifold up to diffeomorphism. A closed $4$-manifold has no unique Kirby diagram, handle pair creation/cancellation, and handle sliding form a complete set of moves that allows one to go from one diagram to another of the same manifold~\cite{Cerf}. For an exposition of \Kds and \Hds the reader is referred to the book~\cite{gs} by Gompf and Stipsicz.

\begin{repalgorithm}{algorithm3}
    Given a \Hd of $M$ and the image under $\varphi$ of each \handle{2} in $M$, one obtains a \Kd of the open book $\operatorname{Ob}(M,\varphi)$ by applying the following algorithm:
    \begin{enumerate}
        \item Replace every \handle{1} attaching region with a pair of $D^3$'s.
        \item\label{2} Add blackboard framing to every \handle{2} attaching sphere.
        \item Add behind the blackboard plane a half-meridian with blackboard framing to each \handle{2} attaching sphere. 
        \item Add in front of the blackboard plane a half-meridian with blackboard framing to each \handle{2} attaching sphere, then replace them with their images under the monodromy $\varphi$.
    \end{enumerate}
\end{repalgorithm}

The result of step~\ref{2} is a \Kd of a half open book, see Algorithm~\ref{algorithm2}. To place the image under $\varphi$ of the \handles{2} into the \Kd of the half open book, we identify each half-space in $\mathbb{R}^3$ separated by the blackboard with the page manifold, which will be explained in Proposition~\ref{cocore is half-meridian}.

\begin{figure}[ht]
    \centering
    \includegraphics[width=0.95\textwidth]{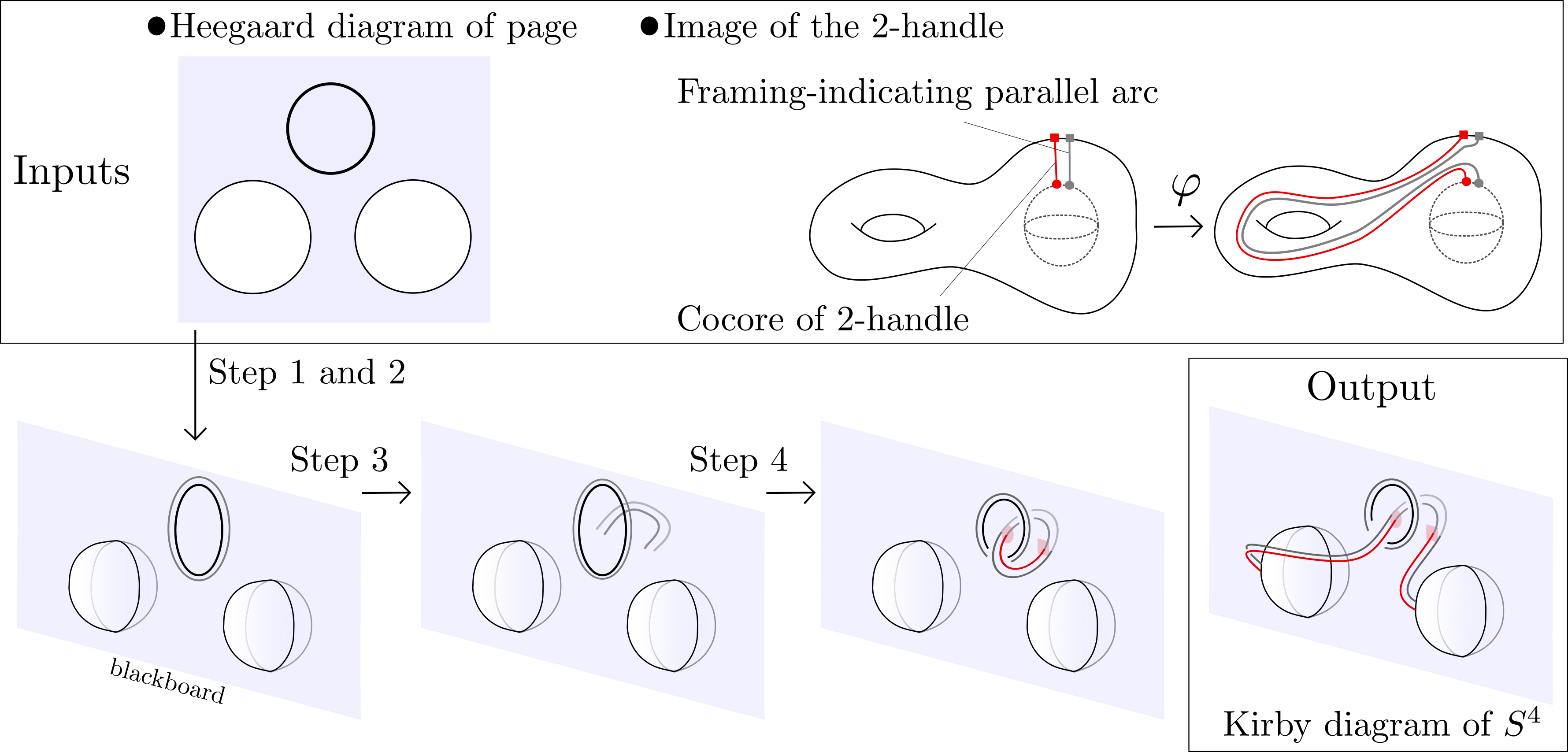}
    \caption{Algorithm~\ref{algorithm3}.}
   \label{fig:first example}
\end{figure}

In Figure~\ref{fig:first example} the \Kd obtained using Algorithm~\ref{algorithm3}, with page punctured solid torus and monodromy torus twist (Defintion~\ref{tt}) as input, consists of a \handle{1} whose belt sphere intersects the attaching sphere of a \handle{2} and another \handle{2} whose belt sphere intersects the attaching sphere of a \handle{3}. Handle cancellation results in a blank diagram, a \Kd of $S^4$, where a single \handle{4} is glued to the \handle{0}. This implies that $\operatorname{Ob}(M',\varphi')$ is diffeomorphic to $\operatorname{Ob}(M,\varphi)\#S^4$, where $M'$ is obtained from $M$ by taking a connected sum with a punctured solid torus and $\varphi'$ is the composition of $\varphi$ with a trivial extension of a torus twist. In particular, Figure~\ref{fig:first example} exhibits a $4$-dimensional open book stabilization as mentioned in~\cite{kegel2023trisecting}.

\Kds of spun $4$-manifolds, which are open books with binding $S^2$, are constructed in~\cite{Montesinos}. An algorithm to obtain a trisection from a Kirby diagram is presented in~\cite{Kepplinger_2022}. For a practical algorithm obtaining a trisection diagram from a $4$-dimensional open book see~\cite{kegel2023trisecting}.

\subsection*{Convention}
Throughout this paper, all maps are smooth and all manifolds are compact, oriented, connected, and smooth. There is a preferred way to smooth corners~\cite[Remark~1.3.3]{gs}, hence we do distinguish between manifolds with and without corners. Concerning handle decomposition of $3$-manifolds with non-empty boundary, we always assume that there is a unique \handle{0} and no \handles{3}~\cite[p.~112]{gs}, and that handles are attached in order of non-decreasing indices to the boundary of handles of lower indices.

\subsection*{Acknowledgment}
I am indebted to Marc Kegel, Felix Schm\"aschke, and Thomas Vogel for helpful conversations. I would like to express my gratitude toward my advisor Marc Kegel for suggesting this topic to me and for his valuable mentoring which helped me to thrive for improvement. I am also grateful to Thomas Vogel for his encouragement. Finally, I would like to warmly thank Li-Chun Wang and give her credit for the 3D graphics in Figure~\ref{fig:hd example dehn}.

\section{Background}\label{background}
An open book is a geometric structure that arises naturally from many fields of mathematics. The application of open books includes the study of fibered knots~\cite{3ob}, zero sets of complex polynomials~\cite{milnor2016singular}, spun embeddings~\cite{pancholi2018embeddings}, and contact structures~\cite{Giroux}, just to name a few.

An \textit{open book decomposition} $(B,\pi)$ of a closed $n$-manifold $X$ consists of a binding $B$ and a fibration $\pi\colon X\backslash B\rightarrow S^1$. The \textit{binding} is a non-empty, not necessarily connected, codimension two submanifold of $X$ with trivial normal bundle. On a neighborhood $N(B)\subset X$ of $B$ there exists a trivialization $N(B)\approx B\times D^2$ using polar coordinates $(r,\phi)$ on $D^2$, such that $\pi|_{N(B)\backslash B}=\phi$. For each $t\in S^1$, the preimage $\pi^{-1}(t)$ is called a \textit{fiber} and its compactification $\overline{\pi^{-1}(t)}$ is called a \textit{page} of the open book.

An \textit{abstract open book decomposition} $(M,\varphi)$ of a closed $n$-manifold $X$ consists of a page $M$ and a monodromy $\varphi$.
The \textit{page} is a compact $(n-1)$-manifold $M$ with non-empty boundary and the \textit{monodromy} is an orientation-preserving self-diffeomorphism on $M$ that equals the identity map in a neighborhood of the boundary $\partial M$. $X$ is diffeomorphic to the closed $n$-manifold $$M\times [0,1]/\sim_\varphi,$$ denoted by $\operatorname{Ob}(M,\varphi)$, where the equivalence relation is given by $$(x,1)\sim_\varphi(\varphi(x),0)\text{ for all } x\in M,\text{ and}$$ $$(x,t)\sim_\varphi(x,t')\text{ for all } x\in \partial M, t,t'\in [0,1].$$

An abstract open book decomposition $(M,\varphi)$ gives rise to an open book decomposition whose binding is $\partial M\times [0,\frac{1}{2}]/\sim_{\varphi}$ and whose pages are $M\times\{t\}$. Since \Kds are defined up to diffeomorphism, we do not need to distinguish between abstract and non-abstract open book decompositions. More details can be found in~\cite{Colin} and~\cite[Remark~2.6]{Etnyre}.

In dimension $3$ every closed manifold has open book decompositions~\cite{3ob}. All closed, simply-connected manifolds of dimension $\geq 6$ with vanishing signature admit open book decompositions~\cite{Winkelnkemper}, and all closed manifolds of odd dimension $\geq 7$ admit open book decompositions~\cite{Lawson}. Quinn proved the state-of-the-art theorem of open books showing that an $n$-manifold admits an open book decomposition if and only if its signature and its Quinn-invariant vanish, where $n\neq 4$~\cite{Quinn}. The Quinn-invariant, also called the asymmetric signature, of a simply connected manifold vanishes if and only if the usual signature does. It is extracted from a version of a middle-dimensional intersection form with group ring coefficients and takes values in a certain Witt group of such forms. This obstruction class vanishes if and only if the corresponding form admits stably a Lagrangian. For further details on this, we refer to the original works~\cite{Quinn,Ranicki}. 

It remains unknown whether these obstructions are complete for the case $n=4$. Thus we would like to initiate a study of $4$-dimensional open books using a tool exclusive to dimension $4$ -- Kirby diagrams.

\section{Kirby diagrams of half open books}
An open book is given by two copies of \textit{half open books} glued together. Thus, a \Kd of an open book with page $M$ can be obtained from a \Kd of a half open book, with the same page, by adding a framed link. The main goal of this section is to present Algorithm~\ref{algorithm1} for constructing a \Kd of half open books. We begin by introducing our non-standard definitions.

Let $M$ be a compact $(n-1)$-manifold with non-empty boundary.
\begin{definition}
    The \textit{half open book with page} $M$ is the $n$-manifold $M\times [0,\frac{1}{2}]/\sim_\frac{1}{2}$, where the equivalence relation is given by $$(x,t)\sim_\frac{1}{2}(x,t') \text{ for all } x\in \partial M, t,t'\in [0,1/2].$$ Call $M\times \{0\}$ and $M\times \{\frac{1}{2}\}\subset \partial(M\times [0,\frac{1}{2}]/\sim_\frac{1}{2})$ the \textit{front} and \textit{back cover} of the half open book, respectively. The codimension two submanifold $\partial M$ of the half open book is called the \textit{binding}.
\end{definition}
The \textit{double} of an $n$-manifold $N$ with non-empty boundary, denoted by $DN$, is an $n$-manifold without boundary obtained from $N$ by gluing $-N$ to it along the boundary using the identity map on $\partial N$. The boundary of the half open book with page $M$, which consists of the front and back cover glued along the binding, is diffeomorphic to $DM$.

\begin{definition}\label{to open hob}
    Given a monodromy $\varphi\colon M\rightarrow M$, we \textit{glue} two half open books with page $M$ \textit{with} $\varphi$ as follows: \begin{itemize}
        \item glue the front cover of the second copy to the back cover of the first copy using the identity map $\operatorname{id}$ and
        \item glue the back cover of the second copy to the front cover of the first copy using the monodromy map $\varphi$.
    \end{itemize} Let $D_{\varphi}(M\times[0,\frac{1}{2}]/\sim_{\frac{1}{2}})$ denote the result of gluing two half open books with page $M$ with $\varphi$.
\end{definition}

\begin{remark}\label{open with id}
    Definition~\ref{to open hob} is well-defined since a monodromy map is equal to the identity map in a neighborhood of $\partial M$. The result of gluing two half open books $D_{\operatorname{id}}(M\times[0,\frac{1}{2}]/\sim_{\frac{1}{2}})$ with the identity map is the double $D(M\times[0,\frac{1}{2}]/\sim_{\frac{1}{2}})$ of the half open book. This explains our choice of notation.
\end{remark}

\begin{figure}[ht]
    \centering
    \includegraphics[width=0.9\textwidth]{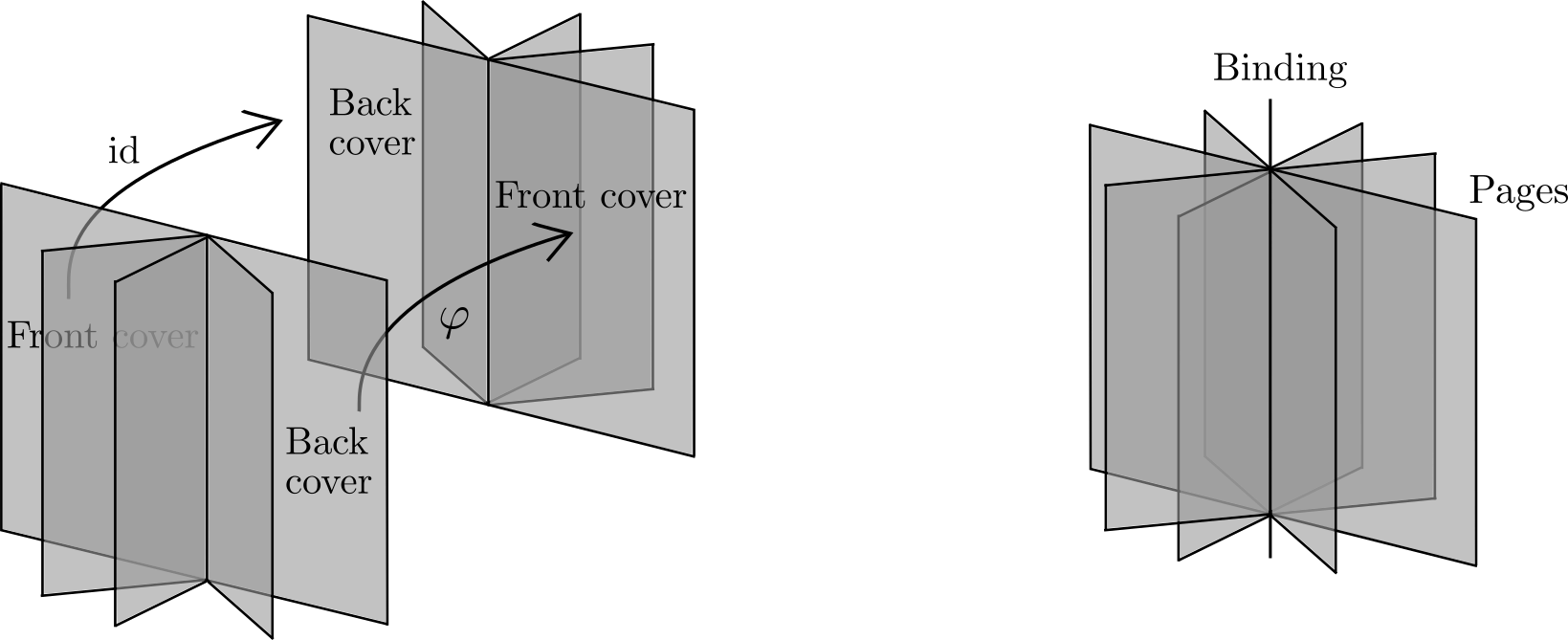}
    \caption{Gluing two half open books gives an open book.}
    
\end{figure}
\begin{proposition}\label{prop}
    The result of gluing two half open books is an open book, more explicitly $D_{\varphi}(M\times [0,\frac{1}{2}]/\sim_\frac{1}{2})$ is diffeomorphic to $\operatorname{Ob}(M,\varphi)$.
\end{proposition}
\begin{proof}
    The proof follows by definition. Glue to the back cover of $M\times [0,\frac{1}{2}]/\sim_\frac{1}{2}$ the front cover of $M\times [\frac{1}{2},1]/\sim_\frac{1}{2}$ with the identity map, we get $M\times [0,1]/\sim$, where $(x,t)\sim(x,t')\text{ for all } x\in \partial M, t,t'\in [0,1]$. Gluing the back cover of $M\times [\frac{1}{2},1]/\sim_\frac{1}{2}$ to the front cover of $M\times [0,\frac{1}{2}]/\sim_\frac{1}{2}$ using $\varphi$ is equivalent to identifying $(x,1)$ with $(\varphi(x),0)$ for all $(x,1)\in M\times [0,1]/\sim$, which gives $\operatorname{Ob}(M,\varphi)$.
\end{proof}
\begin{example}
    The half open book with page $[0,1]$ is given by identifying $(0,0)\in [0,1]\times [0,\frac{1}{2}]$ with $(0,t)$ for all $t\in[0,\frac{1}{2}]$ and identifying $(1,0)$ with $(1,t)$ for all $t\in[0,\frac{1}{2}]$. It is diffeomorphic to $D^2$ and its boundary, the double of $[0,1]$, is diffeomorphic to $S^1$. By Proposition~\ref{prop}, the open book $\operatorname{Ob}([0,1],\operatorname{id})$ is diffeomorphic to $D(D^2)=S^2$.
\end{example}

\begin{proposition}\label{induced handle decomposition}
    A handle decomposition of an $(n-1)$-dimensional manifold $M$ induces a handle decomposition of the $n$-dimensional half open book $M\times [0,\frac{1}{2}]/\sim_\frac{1}{2}$.
\end{proposition}
\begin{proof}
    Firstly, we construct a handle decomposition of $M\times [0,\frac{1}{2}]$ from a handle decomposition of $M$. Suppose $M$ is diffeomorphic to the result $$D^{n-1}\cup_{\phi_1} h_1 \dots\cup_{\phi_m} h_m$$ of attaching the handle $h_i$ via the attaching map $\phi_i$ to $\partial D^{n-1}$ for $i=1,\dots,m$. Here the subscript $i$ is not the index of the handle but an enumeration. To begin with, the $(n-1)$-dimensional \handle{0} $D^{n-1}$ of $M$ induces an $n$-dimensional \handle{0} of $M\times[0,\frac{1}{2}]$. Now assume that the first $i$ handles of $M$ have induced $i$-many handles of $M\times[0,\frac{1}{2}]$, i.e.\ assume we have obtained a handle decomposition of $M_{i-1}\times [0,\frac{1}{2}]$, where $M_{i-1}=D^{n-1}\cup_{\varphi_1}h_1\dots\cup_{\phi_{i-1}}h_{i-1}$ and $0<i-1<m$. Let $\phi_i\colon \partial D^k\times D^{(n-1)-k}\rightarrow \partial M_{i-1}$ be the attaching map of a $(n-1)$-dimensional \handle{k} $h_i$. Then it induces an $n$-dimensional \handle{k} $h_i'$ of $M\times [0,\frac{1}{2}]$ with attaching map $$\phi_i'\colon\partial D^k\times D^{(n-1)-k}\times [0,1/2]\rightarrow \partial(M_{i-1}\times [0,1/2])$$ given componentwise by $$ (\phi_i,\operatorname{id}_{[0,\frac{1}{2}]})\colon\partial D^k\times D^{(n-1)-k}\times [0,1/2]\rightarrow\partial M_{i-1}\times [0,1/2]\subset \partial\left(M_{i-1}\times [0,1/2]\right),$$ such that $M_i\times[0,\frac{1}{2}]$ is diffeomorphic to $(M_{i-1}\times [0,\frac{1}{2}]) \cup_{\phi_i'}h_i'$. Note that $D^{(n-1)-k}\times [0,\frac{1}{2}]$ is diffeomorphic to $D^{n-k}$. By induction, we obtain a handle decomposition of $M\times[0,\frac{1}{2}]$.

    Secondly, we show that $\sim_{\frac{1}{2}}$ preserves the handle decomposition, i.e.\ the handle decomposition on $M\times [0,\frac{1}{2}]$ induces a handle decomposition on $M\times [0,\frac{1}{2}]/\sim_\frac{1}{2}$. Let us look closely at some \handle{k} $h_i'$ in $M\times [0,\frac{1}{2}]$. If $\phi_i'(\partial D^k \times D^{(n-1)-k}\times [0,\frac{1}{2}])\cap \partial (M\times [0,\frac{1}{2}])=\emptyset$, then no identification will take place on this handle. Then by using the identification $D^{(n-1)-k}\times [0,\frac{1}{2}]\approx D^{n-k}$, $\phi_i'$ can be seen as an attaching map of an induced $n$-dimensional \handle{k} in $(M\times [0,\frac{1}{2}])/\sim_\frac{1}{2}$ and we are done. The points $\phi_i'(\partial D^k \times \operatorname{int}(D^{(n-1)-k})\times [0,\frac{1}{2}])\subset \operatorname{int}(M)\times [0,\frac{1}{2}]$ always remain distinct points in the quotient of $\sim_{\frac{1}{2}}$. Hence if $\phi_i'((\partial D^k)\times D^{(n-1)-k}\times [0,\frac{1}{2}])\cap \partial (M\times [0,\frac{1}{2}]) \neq\emptyset$, then $\phi_i'^{-1}(\partial (M\times [0,\frac{1}{2}]))\subseteq\partial D^k \times \partial D^{(n-1)-k} \times [0,\frac{1}{2}]$ and points $(p,t)\in \phi_i'(\partial D^k \times \partial D^{(n-1)-k} \times [0,\frac{1}{2}])=\phi_i(\partial D^k \times \partial D^{(n-1)-k})\times [0,\frac{1}{2}]$ get identified for all $t\in [0,\frac{1}{2}]$. We can identify these points in the preimage already, thus identify $(x,y,t)\in \partial D^k \times \partial D^{(n-1)-k} \times [0,\frac{1}{2}]$ for all $t\in [0,\frac{1}{2}]$ and abusively denote the quotient space by $$\left(\partial D^k \times \partial D^{(n-1)-k} \times [0,1/2]\right)/\sim_{\frac{1}{2}}.$$ Let $\phi_i''$ denote the map induced by $\phi_i'$, then $$\phi_i''\left((\partial D^k \times \partial D^{(n-1)-k} \times [0,1/2])/\sim_{\frac{1}{2}}\right)$$$$=\left(\phi_i'(\partial D^k \times \partial D^{(n-1)-k} \times [0,1/2])\right)/\sim_{\frac{1}{2}}.$$ This implies \begin{multline*} \phi_i''\left((\partial D^k \times D^{(n-1)-k}\times [0,1/2])/\sim_{\frac{1}{2}}\right)=\left(\phi_i'(\partial D^k \times D^{(n-1)-k}\times [0,1/2])\right)/\sim_{\frac{1}{2}}\\ \subset \partial(M_{i-1}\times[0,1/2])/\sim_{\frac{1}{2}}.
    \end{multline*} To see that $\phi_i''$ is an $n$-dimensional \handle{k} in $(M\times[0,\frac{1}{2}])/\sim_{\frac{1}{2}}$, it remains to show $\left(\partial D^k \times D^{(n-1)-k}\times [0,\frac{1}{2}]\right)/\sim_{\frac{1}{2}}$ is diffeomorphic to $\partial D^k \times D^{n-k}$.

    It suffices to observe that $\left(D^{(n-1)-k}\times [0,\frac{1}{2}]\right)/\sim_{\frac{1}{2}}$ is diffeomorphic to $D^{n-k}$, where $(y,t)\sim_{\frac{1}{2}}(y,t')$ for all $y\in \partial D^{(n-1)-k}$, $t$, $t'\in [0,\frac{1}{2}]$. Map the equivalence class $[D^{(n-1)-k}\times \{0\}]$ to the upper hemisphere of $D^{n-k}$, the equivalence class $[D^{(n-1)-k}\times \{\frac{1}{4}\}]$ to the disk bounded by the equator, the equivalence class $[D^{(n-1)-k}\times \{\frac{1}{2}\}]$ to the lower hemisphere, and everything in between correspondingly.
\end{proof}

\begin{figure}[ht]
    \centering
    \begin{tikzpicture}
    \node[] (a) at (3.5,1.1) {attaching region};
    \node[] (b) at (2.3,0.3){};
    \node[] (c) at (0.4,0.3){};
    \node[] (d) at (1,-0.7) {boundary};
    \node[] (e) at (1.4,0){};
    \node[] (f) at (1.4,0.6){};
    \node[red] at (0.14,0.5){$\phi$};
    \draw[blue, thick]  (0.5,0.6) -- (2.5,0.6) ;
    \draw[black, thick] (2,0) -- (0,0);
    \draw[red, thick](0,0) -- (0.5,0.6);
    \draw[red, thick](2.5,0.6) -- (2,0);
    \graph{(a)->[bend left](b)};
    \graph{(a)->[bend right](c)};
    \graph{(d)->[bend right](e)};
    \graph{(d)->[bend left](f)};
    \fill [fill=red] (5,0.7) -- (5.5,1.3) --(5.5,0.3)--(5,-0.3) ;
    \draw [black,densely dotted] (5,0.7) -- (5.5,1.3) --(5.5,0.3)--
    (5,-0.3) --cycle;
    \draw[blue, thick]  (5.5,0.3) -- (7.5,0.3) ;
    \draw[black, thick] (7,-0.3) -- (5,-0.3);
            
    \draw[blue, thick]  (5.5,1.3) -- (7.5,1.3) ;
    \draw[black, thick] (7,0.7) -- (5,0.7);
            
    \fill [fill=red] (7,0.7) -- (7.5,1.3) --(7.5,0.3)--(7,-0.3) ;
    \draw[black,densely dotted] (7,0.7) -- (7.5,1.3) --(7.5,0.3)--
    (7,-0.3) --cycle;   
    \node[] (g) at (4.9,0){};
    \node[] (h) at (7.3,1){};
    \graph{(a)->[bend right](g)};
    \graph{(a)->[bend left](h)};
    \node[red] at (4.8,0.5){$\phi'$};
    \fill[red, rotate around={-15:(8.7,0.5)}] (8.7,0.5) ellipse (0.3 and 0.7);
    \draw[black,densely dotted, rotate around={-15:(8.7,0.5)}] (8.7,0.5) ellipse (0.3 and 0.7);
    \node[] (i) at (8.4,-0.15){};
    \node[] (j) at (10.7,-0.15){};
    \node[] (k) at (8.7,1.15){};
    \node[] (l) at (10.9,1.15){};
    \node[] (m) at (8.25,0.4){};
    \node[] (n) at (10.5,0.4){};
    \node[] (o) at (8.9,0.55){};
    \node[] (p) at (10.9,0.55){};
    \node[red] at (8.25,0.8) {$\phi''$};
    \draw[black, thick] (m)--(n);
    \draw[black, densely dotted] (i)--(j);
    \draw[black, densely dotted] (k)--(l);
    \draw[blue,thick] (o)--(p);
    \fill[red, rotate around={-15:(10.7,0.5)}] (10.7,0.5) ellipse (0.3 and 0.7);
    \draw[black,densely dotted, rotate around={-15:(10.7,0.5)}] (10.7,0.5) ellipse (0.3 and 0.7);
    
    \node[] (x) at (1,-1.3) {$D^1\times D^1$};
    \node[] (y1) at (6.1,-0.9) {$D^1\times D^1\times[0,\frac{1}{2}]$};
    \node[] (y2) at (6.1,-1.4) {$\approx D^1\times D^2$};
    \node[] (z1) at (10,-0.9) {$\left(D^1\times D^1\times[0,\frac{1}{2}]\right)/\sim_\frac{1}{2}$};
    \node[] (z2) at (10,-1.4){$\approx D^1\times D^2$};
    \end{tikzpicture}
    \caption{A $2$-dimensional \handle{1} inducing a $3$-dimensional \handle{1}.}
    \label{fig:induced handle example}
\end{figure}

In Figure~\ref{fig:induced handle example}, we assume the black and blue edges of the \handle{1} $D^1\times D^1$ are in the boundary of $M$, so the rectangular faces spanned by them in the product $M\times [0,\frac{1}{2}]$ will each be collapsed to an edge in the half open book $M\times [0,\frac{1}{2}]/\sim_{\frac{1}{2}}$.

From now on let $M$ be a compact $3$-manifold with non-empty boundary.

\begin{algorithm}\label{algorithm1}
Given a \Hd of $M$, one obtains a \Kd of the half open book with page $M$ by applying the following algorithm:
    \begin{enumerate}
        \item Place the given \Hd in the $yz$-plane in $\mathbb{R}^3$, assuming the positive $x$-direction points out of the paper.
        \item\label{algo1 step2} Replace every \handle{1} attaching region with a pair of $D^3$'s as shown in Figure~\ref{fig:alignment}.
        \item Add in the $yz$-plane a parallel knot to each \handle{2} attaching sphere to indicate the blackboard framing.
    \end{enumerate}
\end{algorithm}
\begin{figure}[ht]
    \centering
    \includegraphics[width=0.5\textwidth]{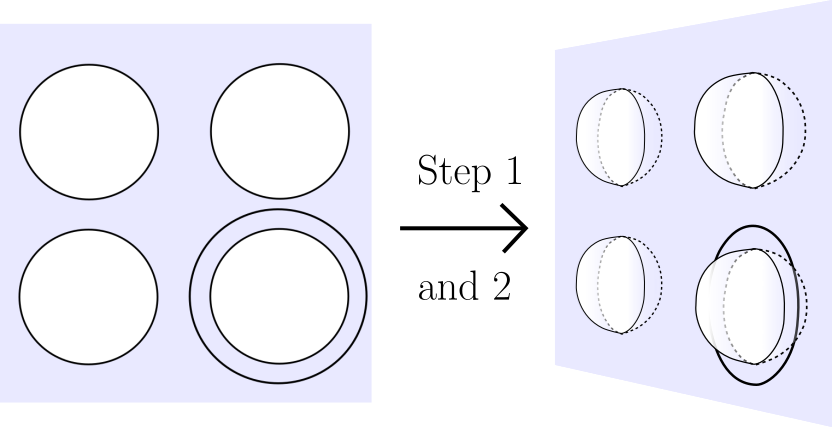}
\caption{Replace every \handle{1} attaching region in the given \Hd with a pair of $D^3$'s.}
\label{fig:alignment}
\end{figure}

\begin{proof}
      By Proposition~\ref{induced handle decomposition}, $M\times [0,\frac{1}{2}]$ and $M\times [0,\frac{1}{2}]/\sim_{\frac{1}{2}}$ admit equivalent handle decompositions and are diffeomorphic to each other, thus $M\times [0,\frac{1}{2}]$ and $M\times [0,\frac{1}{2}]/\sim_{\frac{1}{2}}$ share an equivalent Kirby diagram. Let us construct a \Kd of the former one. A handle decomposition of $M$ induces a handle decomposition of $M\times [0,\frac{1}{2}]$ in the following way:
     \begin{itemize}
        \item A $3$-dimensional \handle{0} induces a $4$-dimensional \handle{0}. So we will now attach handles to $\partial D^4-\{$point$\}\approx\mathbb{R}^3$ instead of to $\partial D^3-\{$point$\}\approx\mathbb{R}^2$. Since \Hds are defined in $\mathbb{R}^2$, they can be embedded in the $yz$-plane in $\mathbb{R}^3$.
        \item A $3$-dimensional \handle{1} induces a $4$-dimensional \handle{1}. Thus the attaching region of a \handle{1} becomes a pair of $D^3$'s.
        \item A $3$-dimensional \handle{2} induces a $4$-dimensional \handle{2} with framing given by the product structure of $M\times[0,\frac{1}{2}]$, which coincides with the blackboard framing.
    \end{itemize}
\end{proof}

\section{Kirby diagrams of 4-dimensional open books}
By assumption, any $3$-dimensional page and therefore any $4$-dimensional half open book, with the induced handle decomposition (Proposition~\ref{induced handle decomposition}), has no $3$-handles. Gluing a half open book with one \handle{0}, $x$ \handles{1}, and $y$ \handles{2} to a $4$-manifold $X$ along common boundary introduces $y$ \handles{2}, $x$ \handles{3}, and one \handle{4}, i.e.\ adds a $y$-component framed link to a \Kd of $X$. Since an open book is obtained by gluing two half open books, we add a framed link to a \Kd of the half open book to obtain a \Kd of the open book.

\subsection{Trivial monodromy} Let $M$ be the page, a compact $3$-manifold with non-empty boundary, as before. We denote the identity map on the page by $\operatorname{id}$.
\begin{algorithm}\label{algorithm2}
Given a \Hd of $M$, one obtains a \Kd of the open book $\operatorname{Ob}(M,\operatorname{id})$ by applying the following algorithm:
   \begin{enumerate}
    \item\label{algo2 step1} Use Algorithm~\ref{algorithm1} to get a \Kd of the half open book with page $M$.
    \item\label{add 0 framed meridian} Add a $0$-framed meridian to each blackboard-framed \handle{2} attaching sphere in the \Kd of the half open book. 
   \end{enumerate}
\end{algorithm}
\begin{proof}
    By Proposition~\ref{prop} and Remark~\ref{open with id}, $\operatorname{Ob}(M,\varphi)$ is diffeomorphic to $D_{\varphi}(M\times [0,\frac{1}{2}]/\sim_\frac{1}{2})=D(M\times [0,\frac{1}{2}]/\sim_\frac{1}{2})$ when the monodromy $\varphi$ is (isotopic to) the identity map. For completeness, we summarize the proof given in~\cite{gs} showing that adding a $0$-framed meridian to each \handle{2} attaching sphere in a \Kd gives a \Kd of the double.
    
    A handle decomposition on $DX=X\cup$handles is obtained by adding handles to $X$, where $X=M\times [0,\frac{1}{2}]/\sim_{\frac{1}{2}}$. It suffices to understand the induced \handles{2}. A \handle{2} $h$ in $X$ induces a \handle{2} $h'$ in $DX$ with the roles of core and cocore reversed. Since we are gluing using the identity map, the attaching map of an induced \handle{2} $\varphi\colon \partial D^2\times D^2\rightarrow D^2\times \partial D^2\subset \partial X$ is given by interchanging the factors of $\partial D^2\times D^2$. Thus $h'$ is attached along the belt sphere of $h$, which is isotopic to a meridian of the attaching sphere of $h$. The framing of the induced \handle{2} is given by the product framing coming from the product structure of the \handle{2} $D^2\times D^2$, which coincides with the Seifert framing. Therefore to get a \Kd of the open book, add a $0$-framed meridian to each \handle{2} attaching sphere in a \Kd of the half open book.
\end{proof}
\begin{remark}
    The framings of a $4$-dimensional \handle{2} are classified by elements of $\pi_1(O(2))\approx \mathbb{Z}$. When the attaching sphere $\varphi(\partial D^2)$ bounds a Seifert surface, the attaching sphere can be pushed along it to give a parallel knot and hence a canonical choice of the $0$-th framing.
\end{remark}

\subsection{Non-trivial monodromy} 

Under the standard assumptions that handles are attached with non-decreasing indices to the boundary of handles of lower indices and that there is a unique \handle{0} and no \handle{3}, the cocore of a \handle{2} in $M$ is an arc in $M$ with endpoints in $\partial M$.

\begin{proposition}\label{cocore is half-meridian}
    The cocore of a \handle{2} in the back cover of the half open book with page $M$ corresponds to a half-meridian behind the $yz$-plane in a \Kd of the half open book.
\end{proposition}
\begin{proof}
    Recall that a \Hd of $M$ records where $1$- and \handles{2} are attached to the boundary of the \handle{0}. Situate a \Hd in the $yz$-plane of $\mathbb{R}^3$ and view the region in front of the $yz$-plane as the interior of the \handle{0} and the $yz$-plane as the boundary of the \handle{0} (with a point removed). A \handle{2} of $M$ is being attached to $\partial D^3$ from behind the $yz$-plane in this perspective.
    \begin{figure}[ht]
        \centering
        \includegraphics[scale=0.36]{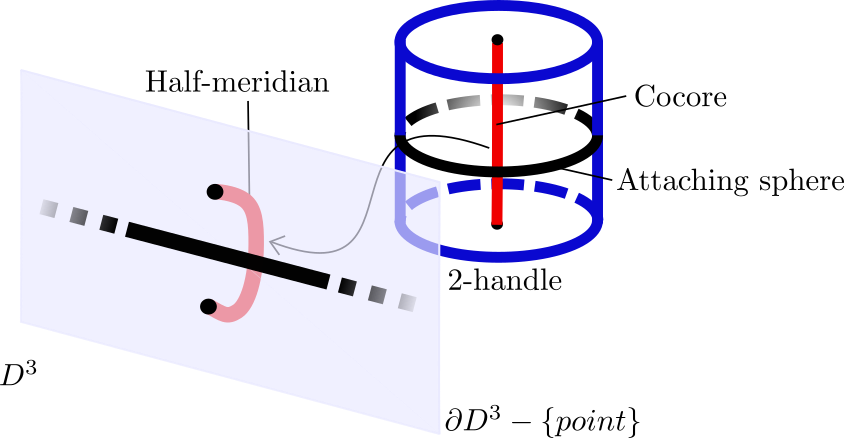}
        \caption{Attaching a \handle{2} along $\partial D^3$.}
        \label{fig:half-meridian is cocore}
    \end{figure}
    
    Now we recall the construction of a \Kd of a half open book. We place a \Hd of the page in the $yz$-plane in $\mathbb{R}^3$, replace each $D^2$ with a $D^3$, and add blackboard framing to each \handle{2} attaching sphere. The blank space in $\mathbb{R}^3$ is the boundary of the half open book, which consists of two copies of $M$, namely the front and back covers. Without loss of generality, we may identify the region behind the $yz$-plane with the back cover and the region in front of the $yz$-plane with the front cover. Thus the cocore of a \handle{2} in the back cover corresponds to a half-meridian, with endpoints on the $yz$-plane, of the \handle{2} attaching sphere as shown in Figure~\ref{fig:half-meridian is cocore}. Similarly, a half-meridian in front of the $yz$-plane corresponds to the cocore of a \handle{2} in the front cover.
\end{proof}

\begin{algorithm}\label{algorithm3}
Given a \Hd of $M$ and the image under the monodromy $\varphi$ of each \handle{2} in $M$, one obtains a \Kd of the open book $\operatorname{Ob}(M,\varphi)$ by applying the following algorithm:
    \begin{enumerate}
        \item\label{algo3 step1} Use Algorithm~\ref{algorithm1} to get a \Kd of the half open book with page $M$.
    \item\label{add half-meridian} Add behind the $yz$-plane a half-meridian with blackboard framing to each \handle{2} attaching sphere in the \Kd of the half open book as in Figure~\ref{fig:half-meridian}. 
        \item\label{algo3 monodromy} Add in front of the $yz$-plane a half-meridian with blackboard framing to each \handle{2} attaching sphere in the \Kd of the half open book, then replace them with their images under $\varphi$.
    \end{enumerate}
\end{algorithm}
\begin{figure}[ht]
        \centering
        \includegraphics[scale=0.5]{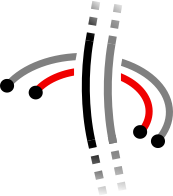}
        \caption{Step~\ref{add half-meridian} of Algorithm~\ref{algorithm3}.}
        
        \label{fig:half-meridian}
    \end{figure}
\begin{proof}
    By Proposition~\ref{prop}, $\operatorname{Ob}(M,\varphi)$ is diffeomorphic to $D_{\varphi}(M\times [0,\frac{1}{2}]/\sim_\frac{1}{2})$, the goal is to construct a \Kd of $D_{\varphi}(M\times [0,\frac{1}{2}]/\sim_\frac{1}{2})$. Recall that the boundary of the half open book $M\times [0,\frac{1}{2}]/\sim_\frac{1}{2}$ consists of a front and back cover, which are diffeomorphic to the page $M$. 
    \begin{figure}[ht]
        \centering
        \includegraphics[scale=0.2]{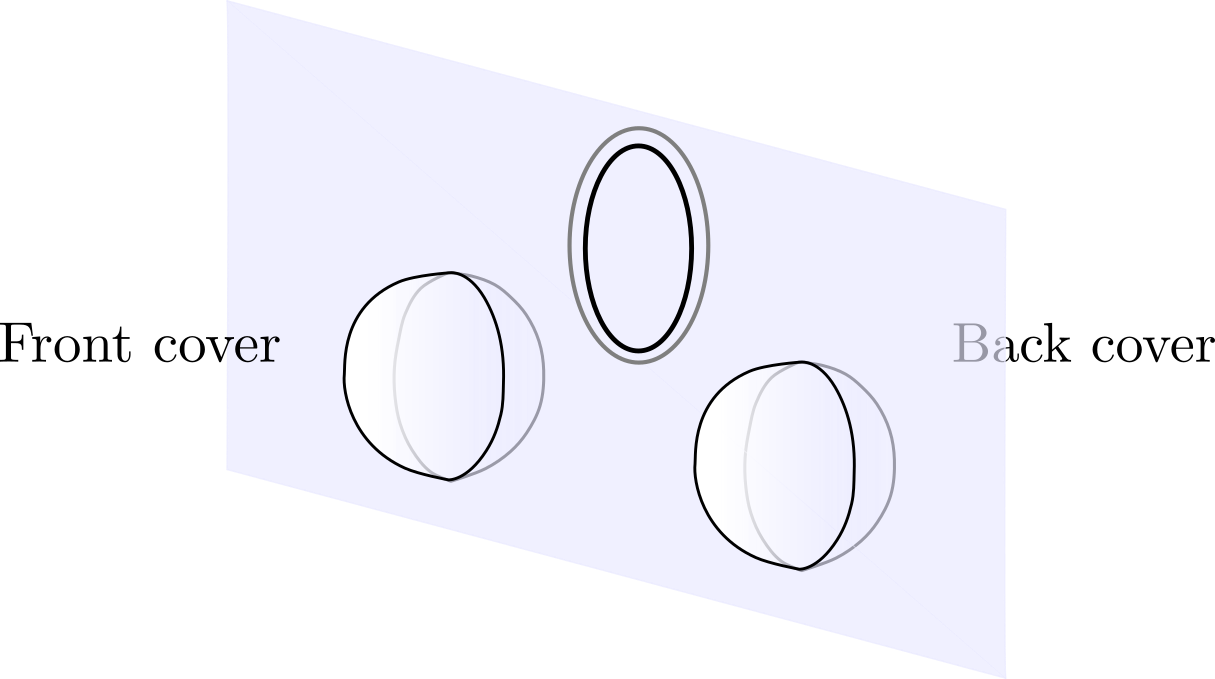}
        \caption{\Kd of half open book (with page punctured solid torus)}
        \label{fig:kd of hob}
    \end{figure}
    Let us focus on the \Kd of the half open book lying in the $yz$-plane in $\mathbb{R}^3$ provided by Algorithm~\ref{algorithm1}. The blank space in a \Kd of a manifold without $3$- and \handles{4} represents a subset of the boundary. We identify the region in front of and behind the $yz$-plane with the front and back cover of the half open book respectively, as described in the proof of Proposition~\ref{cocore is half-meridian}. The $yz$-plane itself, where the front and back covers meet, corresponds to the binding $\partial M\subset M\times [0,\frac{1}{2}]/\sim_{\frac{1}{2}}$.    
    
    Also recall Definition~\ref{to open hob}, $D_{\varphi}(M\times [0,\frac{1}{2}]/\sim_\frac{1}{2})$ is obtained by gluing two copies of $M\times [0,\frac{1}{2}]/\sim_\frac{1}{2}$ along the boundary. Namely,
    \begin{itemize}
        \item glue the front cover of the second copy to the back cover of the first copy with the identity map and
        \item glue the back cover of the second to the front cover of the first with the given monodromy $\varphi$.
    \end{itemize} 
    A handle decomposition on $D_{\varphi}(M\times [0,\frac{1}{2}]/\sim_\frac{1}{2})$ is obtained by adding handles to the half open book $M\times [0,\frac{1}{2}]/\sim_\frac{1}{2}$. It suffices to understand the induced \handles{2}. The attaching sphere of an induced \handle{2} is the union of two "attaching arcs", one in the front and one in the back cover. In the back cover, where the gluing map is the identity map, the attaching arc is given by a half-meridian of the \handle{2} attaching sphere as in the proof of Algorithm~\ref{algorithm2}. In the front cover, the gluing map is given by the monodromy $\varphi$ and not the identity map. Hence instead of a half-meridian, the attaching arc is given by the image under $\varphi$ of the half-meridian. Using the identification of a half-meridian with the cocore of a \handle{2} (Proposition~\ref{cocore is half-meridian}), we conclude that in the front cover, the attaching arc is given by the image of the cocore of the \handle{2}. Similarly, a framing-indicating parallel knot consists of two arcs: a parallel arc of the half-meridian behind the $yz$-plane and the image under $\varphi$ of a parallel arc of the half-meridian (or cocore of \handle{2}) in the other half-space.
\end{proof}

\begin{figure}[ht]
    \centering
    \includegraphics[width=0.9\textwidth]{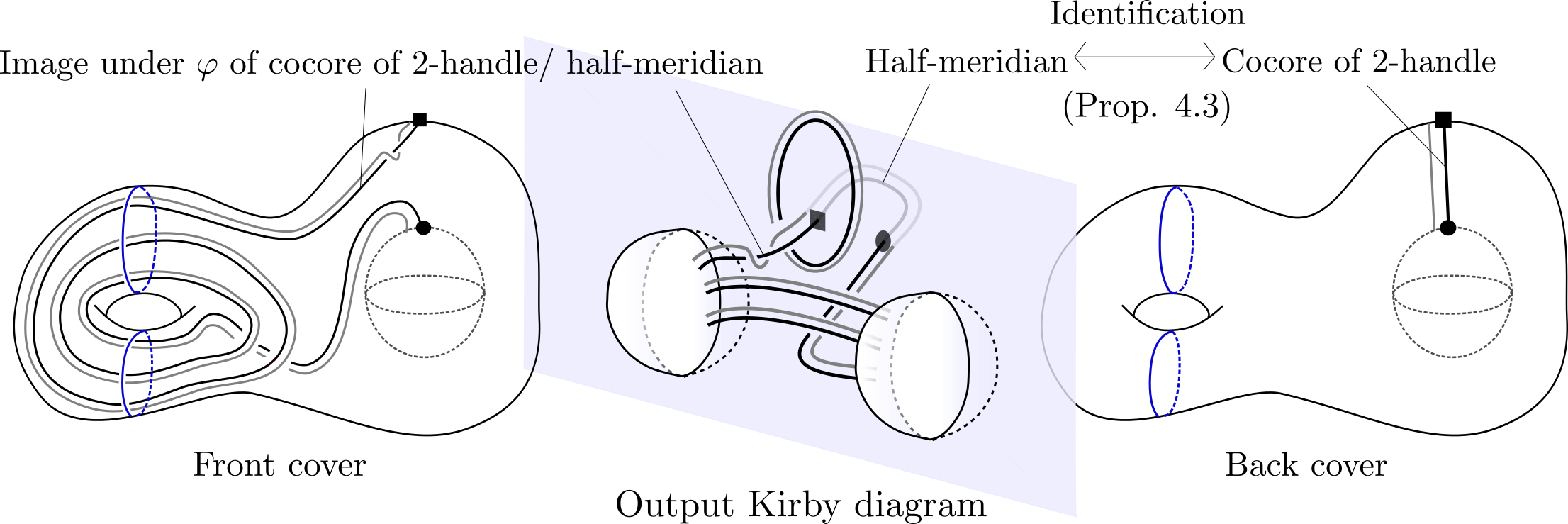}
    \caption{\Kd of open book with page punctured solid torus and monodromy the composition of 3 torus twists (Defintion~\ref{tt}).}
\end{figure}

\section{Examples}\label{examples}
Given $g,n\geq0$, let $H_{g,n}$ denote the \textit{punctured handlebody} obtained from a \handle{0} $D^3$ by attaching $g$ \handles{1} and $n$ \handles{2} to $\partial D^3$. For example, $H_{1,0}$ is a solid torus and $H_{0,1}$ is a punctured ball. When $g=n=1$, the punctured handlebody is also referred to as a punctured solid torus (Figure~\ref{fig:hd of punctured solid torus}). A punctured handlebody has a canonical handle decomposition and hence a canonical Heegaard diagram as shown in Figure~\ref{fig:canonical hd}.
\begin{figure}[ht]
    \centering
    \includegraphics[scale=0.25]{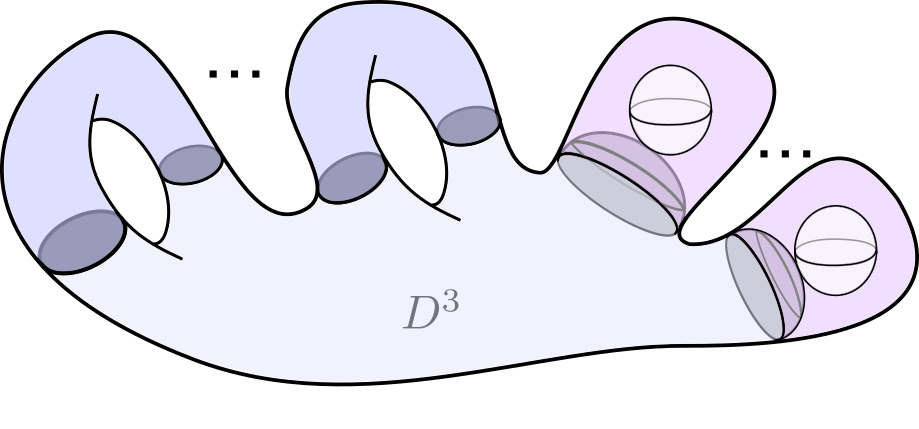}
    
    \includegraphics[scale=0.25]{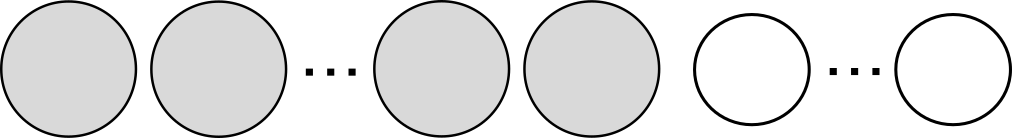}
    \caption{A punctured handlebody and its Heegaard diagram.}
    \label{fig:canonical hd}
\end{figure}

\subsection{Constructing Heegaard diagrams of 3-dimensional open books}
The construction of a \Hd of a $3$-dimensional open book using half open books is more straightforward, thus it is a good idea to see an example in dimension $3$ before returning to dimension $4$. We carefully go through the construction of a \Hd of the $3$-dimensional open book $\operatorname{Ob}(D^2\cup h^1,\psi)$ with page annulus. The method essentially follows~\cite{Honda}.

The half open book $(D^2\cup h^1)\times [0,\frac{1}{2}]/\sim_\frac{1}{2}$ is a solid torus $H_{1,0}$ as visualized in Figure~\ref{fig:page annulus}, whose canonical \Hd consists of a pair of $D^2$'s. Let $L$ be a line in $\mathbb{R}^2$ going through the centers of the pair of $D^2$'s, then the two regions separated by $L$ correspond to the front and back covers of the half open book. We only need to take care of the attaching sphere of the \handle{2} induced by the \handle{1} of $D^2\cup h^1$.

\begin{figure}[ht]
        \centering
        \includegraphics[scale=0.5]{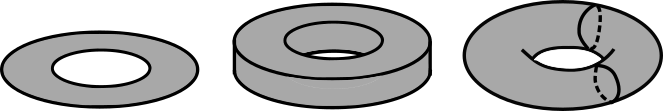}
        \caption{Left to right: $D^2\cup h^1$, $(D^2\cup h^1)\times [0,\frac{1}{2}]$, $(D^2\cup h^1)\times [0,\frac{1}{2}]/\sim_{\frac{1}{2}}$.}
        \label{fig:page annulus}
\end{figure}

When the monodromy $\psi=\operatorname{id}$ is trivial, add a meridian around $D^2$, because the attaching sphere coincides with the belt sphere of the \handle{1}. We obtain a \Hd of $S^1\times S^2$ as illustrated in Figure~\ref{fig:hd example id}. This is expected since $\operatorname{Ob}(D^2\cup h^1,\operatorname{id})$ is diffeomorphic to $DH_{1,0}$, and the double of a solid torus is diffeomorphic to $S^1\times S^2$.

\begin{figure}[ht]
    \centering
    \includegraphics[scale=0.28]{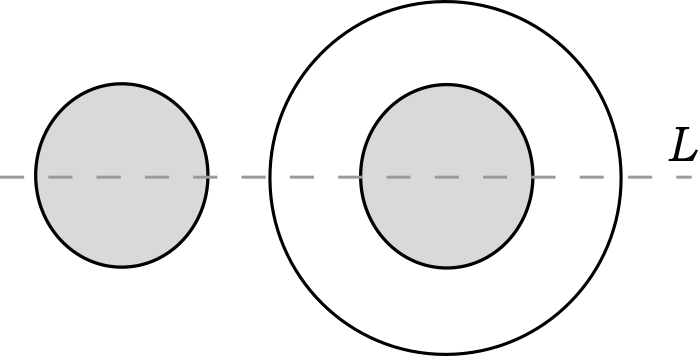}
    \caption{$\operatorname{Ob}(D^2\cup h^1,\operatorname{id})$ is diffeomorphic to $S^1\times S^2$.}
    \label{fig:hd example id} 
\end{figure}

Now we consider the monodromy $\psi$ given by the composition of three Dehn twists along the core of the annulus. The \Hds of $\operatorname{Ob}(D^2\cup h^1,\psi)$ and $\operatorname{Ob}(D^2\cup h^1,\operatorname{id})$ should look the same on one side of $L$ since we always glue using the identity map on the back cover. Without loss of generality, add a half meridian to a $D^2$ below \textit{L}. The non-trivial monodromy only plays a role above $L$ on the front cover of the half open book. Since the image of the cocore of the \handle{1} is an arc in $D^2\cup h^1$ running thrice through the \handle{1}, we add an arc above \textit{L} that intersects the $D^2$ thrice before joining with the half meridian as reflected in Figure~\ref{fig:hd example dehn}. We obtain a \Hd of the lens space $\operatorname{L}(3,1)$, which will be defined in Section~\ref{generalized lens space}.

\begin{figure}[ht]
    \begin{subfigure}{\textwidth}
        \centering
        \includegraphics[scale=0.28]{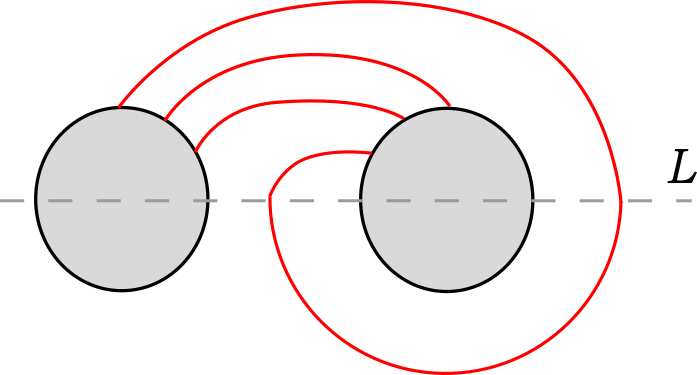}
        
    \end{subfigure}
    \begin{subfigure}{.45\textwidth}
        \centering
        \includegraphics[scale=0.3]{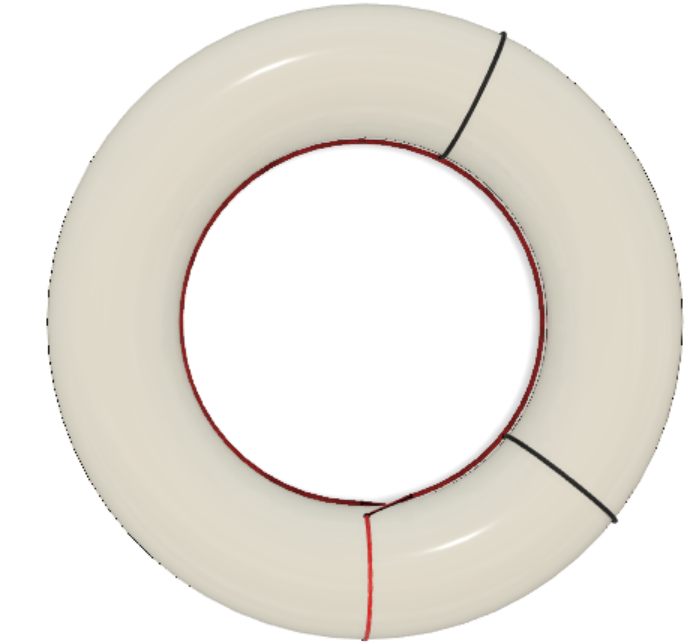}
        
    \end{subfigure}
    \hfill
    \begin{subfigure}{.45\textwidth}
        \centering
        \includegraphics[scale=0.28]{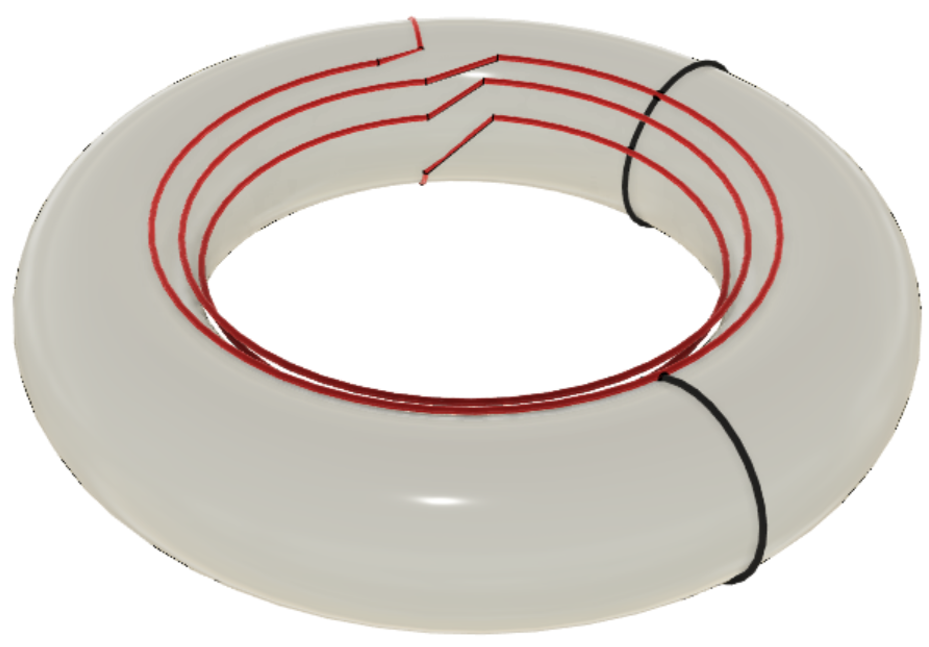}
        
    \end{subfigure}
    \caption{The attaching sphere (red) of the induced \handle{2} is marked on the Heegaard diagram (top), on the back cover (left), and on the front cover (right) of the half open book $H_{1,0}$.}
    \label{fig:hd example dehn}
\end{figure}

\subsection{2-sphere bundles over 2-sphere}\label{hopf link section}

Now we construct a \Kd of the open book with page punctured ball $H_{0,1}$ and trivial monodromy. The punctured ball is obtained by removing an open ball from $D^3$, or it can also be obtained by attaching a single \handle{2} to $D^3$. This handle decomposition gives rise to a \Hd consisting of a single unknot. Using Algorithm~\ref{algorithm1}, we obtain a \Kd of the half open book consisting of an unknot with blackboard framing, which is a \Kd of $D^2\times S^2$~\cite[120]{gs}. Finally, apply Step~\ref{add 0 framed meridian} of Algorithm~\ref{algorithm2} to get a Hopf link with each component labeled $0$ as in Figure~\ref{fig:punctured ball}. It is well known that the output is a \Kd of the trivial $S^2$-bundle over $S^2$~\cite[127]{gs}, thus $\operatorname{Ob}(H_{0,1},\operatorname{id})$ is diffeomorphic to $S^2\times S^2$.
\begin{figure}[ht]
    \begin{subfigure}{0.3\textwidth}
        \centering
        \includegraphics[scale=0.45]{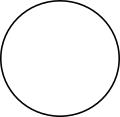}
        \caption{Input: $H_{0,1}$.}   
    \end{subfigure}
    \hfill
    \begin{subfigure}{0.3\textwidth}
        \centering
        \includegraphics[scale=0.45]{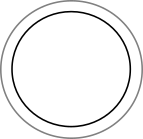}
        \caption{$\operatorname{hob}(H_{0,1})$.}
    \end{subfigure}
    \hfill
    \begin{subfigure}{0.3\textwidth}
        \centering
        \includegraphics[scale=0.45]{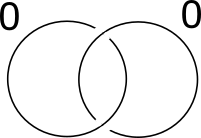}
        \caption{Output: $\operatorname{Ob}(H_{0,1},\operatorname{id})$}
        \label{fig:punctured ball}
    \end{subfigure}
    \caption{Applying Algorithm~\ref{algorithm2}.}
\end{figure}

The \textit{mapping class group} of the punctured ball $H_{0,1}$, denoted by $$\operatorname{MCG}(H_{0,1}, \partial H_{0,1}),$$ is the group of isotopy classes of orientation-preserving diffeomorphisms of the punctured ball that fix the boundary pointwise. Elements in $\operatorname{MCG}(H_{0,1},\partial H_{0,1})$ are equivalent to the isotopy classes of loops in the space $\Omega$ of all orientation-preserving diffeomorphisms of $S^2$. $\Omega$ deformation retracts to the rotation group $\operatorname{SO}(3)$~\cite{Smale}, thus $$\operatorname{MCG}(H_{0,1},\partial H_{0,1})\approx\pi_1(\Omega)\approx\pi_1(\operatorname{SO}(3))\approx\pi_1(\mathbb{RP}^3)\approx\mathbb{Z}/2\mathbb{Z}.$$ Consequently, there can be at most two non-diffeomorphic open books with the punctured ball as page. The generator of $\operatorname{MCG}(H_{0,1},\partial H_{0,1})$ is given by a boundary parallel sphere twist $\sigma$ (Defintion~\ref{st}) as discussed in~\cite[Lemma~2.5]{bbp}. Construction of a \Kd of $\operatorname{Ob}(H_{0,1},\sigma)$ using Algorithm~\ref{algorithm3} is demonstrated in Figure~\ref{fig:punctured ball2}.
\begin{figure}[ht]
     \begin{subfigure}{0.3\textwidth}
        \centering
        \includegraphics[scale=0.45]{punctured_ball_0_1.png}
        \caption{Input: $H_{0,1}$.}
    \end{subfigure}
    \hfill
    \begin{subfigure}{0.3\textwidth}
        \centering
        \includegraphics[scale=0.45]{punctured_ball_0_2.png}
        \caption{$\operatorname{hob}(H_{0,1})$.}
    \end{subfigure}
    \hfill
    \begin{subfigure}{0.3\textwidth}
        \centering
        \includegraphics[scale=0.45]{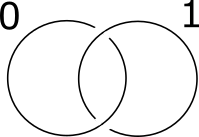}
        \caption{Output: $\operatorname{Ob}(H_{0,1},\sigma)$}
    \end{subfigure}
    \caption{Using Algorithm~\ref{algorithm3} we observe that $\operatorname{Ob}(H_{0,1},\sigma)$ is diffeomorphic to $S^2 \Tilde{\times} S^2$, the twisted $S^2$-bundle over $S^2$.}
    \label{fig:punctured ball2}
    \end{figure}

Let $(\sigma)^n=\sigma\circ\dots\circ\sigma$ denote the composition of $n$ sphere twists, which changes the framing of the cocore by $n$. $(\sigma)^n$ is isotopic to the identity map when $n$ is even and isotopic to a single sphere twist when $n$ is odd because a sphere twist generates $\operatorname{MCG}(H_{0,1},\partial H_{0,1})\approx\mathbb{Z}/2\mathbb{Z}$. According to Algorithm~\ref{algorithm3}, $\operatorname{Ob}(H_{0,1},(\sigma)^n)$ has a \Kd consisting of a Hopf link labeled $0$ and $n$. This shows that the diffeomorphism type of the $4$-manifold represented by the Hopf link labeled $0$ and $n$ only depends on the parity of $n$~\cite[130]{gs}, as summarized by Figure~\ref{fig:exotic punctured ball}.
\begin{figure}[ht]
    \centering
    \includegraphics[scale=0.4]{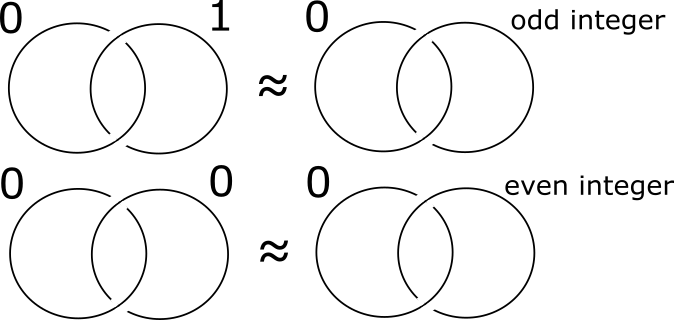}
    \caption{$S^2\Tilde{\times} S^2$ (top) and $S^2\times S^2$ (bottom): the only two open books with the punctured ball as page.}
    \label{fig:exotic punctured ball}
\end{figure}

\subsection{Spun and twist spun lens spaces}\label{generalized lens space}
    Let $p,q$ be relatively prime, non-negative integers. The \textit{lens space} \begin{align*} \operatorname{L}(p,q)&=S^3\big/(\mathbb{Z}/p\mathbb{Z})\\&=\big\{(z_1,z_2)\in\mathbb{C}^2\big| |z_1|^2+|z_2|^2=1\big\}\big/(z_1,z_2)\sim(e^{2\pi i/p}\cdot z_1,e^{2\pi iq/p}\cdot z_2) \end{align*} is a closed $3$-manifold, which can also be obtained by gluing together two solid tori with a diffeomorphism of on the boundary that takes a meridian of one torus to a curve $-p\lambda+q\mu$ on the other torus, where $\lambda$ and $\mu$ denote the longitude and meridian, respectively. The solid torus can be decomposed into a \handle{0} and a \handle{1}. Therefore the lens space can be decomposed into a \handle{0}, a \handle{1}, a \handle{2} and a \handle{3}. (The $2$- and \handles{3} are induced by the $0$- and \handles{1}.) To get a handle decomposition of the punctured lens space $\operatorname{L}(p,q)-D^3$, simply remove the \handle{3}. This handle decomposition leads to a \Hd of $\operatorname{L}(p,q)-D^3$ as depicted in Figure~\ref{fig:hd of lens space}.

Figure~\ref{fig:lens space} demonstrates the construction of a \Kd of the open book $\operatorname{Ob}(\operatorname{L}(3,1)-D^3,\operatorname{id})$ using Algorithm~\ref{algorithm2}. 
\begin{figure}[ht]
    \centering
    \includegraphics[scale=0.25]{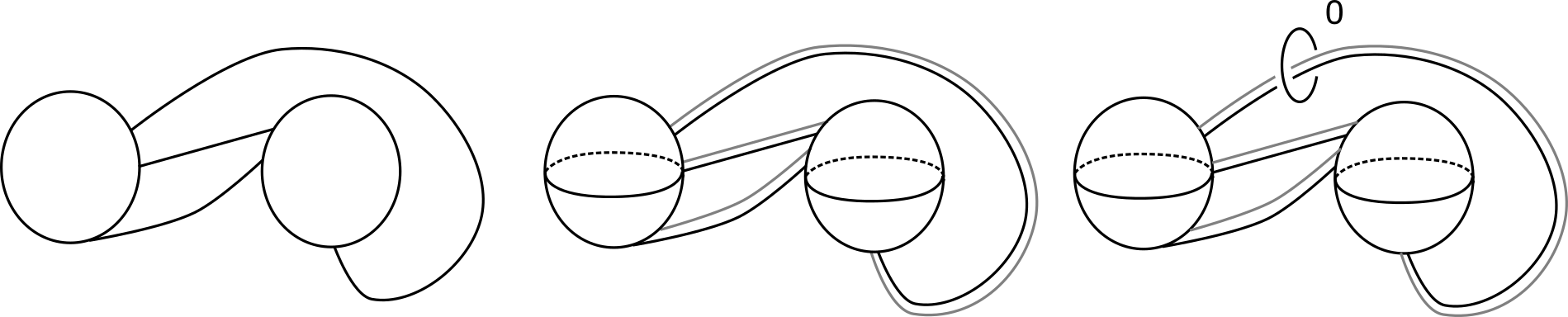}
    \caption{Left to right: input \Hd of $\operatorname{L}(3,1)-D^3$, a \Kd of half open book with page $\operatorname{L}(3,1)-D^3$, output \Kd of the open book $\operatorname{Ob}(\operatorname{L}(3,1)-D^3,\operatorname{id})$. }
    \label{fig:lens space}
\end{figure}

    Given a closed 3-manifold $M$, let $\mathring{M}$ denote $M$ with an open $3$-ball removed. Gordon~\cite{Gordon} defines the \textit{spin} of $M$ to be obtained by gluing $\mathring{M}\times S^1$ to $S^2\times D^2$ using the identity map on the boundary $\operatorname{id}_{S^2\times S^1}$. Observe that the open book $\operatorname{Ob}(\operatorname{L}(p,q)-D^3,\operatorname{id})$ is the spin of a lens space, also known as a \textit{spun} lens space. One obtains a \Kd of a spun lens space as illustrated by Figure~\ref{fig:generalized lens space}.
    
    The \textit{twisted spin} of $M$ is obtained by gluing $\mathring{M}\times S^1$ to $S^2\times D^2$ by the unique self-diffeomorphism of $S^2 \times S^1$ not extending over $S^2 \times D^2$~\cite{gluck}. Since twisted spin of the lens space $\operatorname{L}(p,q)$ is diffeomorphic to $\operatorname{Ob}(\operatorname{L}(p,q)-D^3,\sigma)$, where $\sigma$ is a boundary parallel sphere twist (Defintion~\ref{st}), we can apply Algorithm~\ref{algorithm3}. 
    \begin{figure}[ht]
    \begin{subfigure}{0.3\textwidth}
        \centering
        \includegraphics[width=\textwidth]{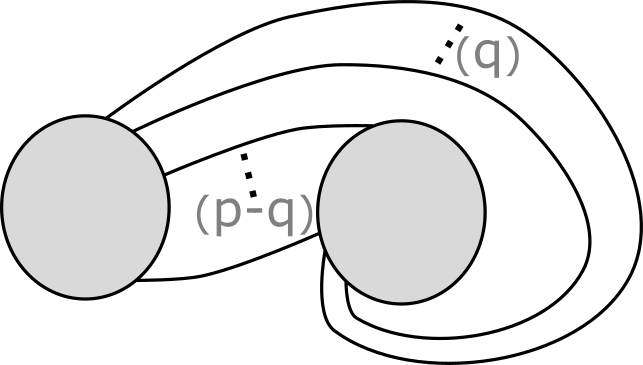}
        \caption{A \Hd of $\operatorname{L}(p,q)-D^3$.}   
    \label{fig:hd of lens space}   
    \end{subfigure}
    \hfill
    \begin{subfigure}{0.3\textwidth}
        \centering
        \includegraphics[width=\textwidth]{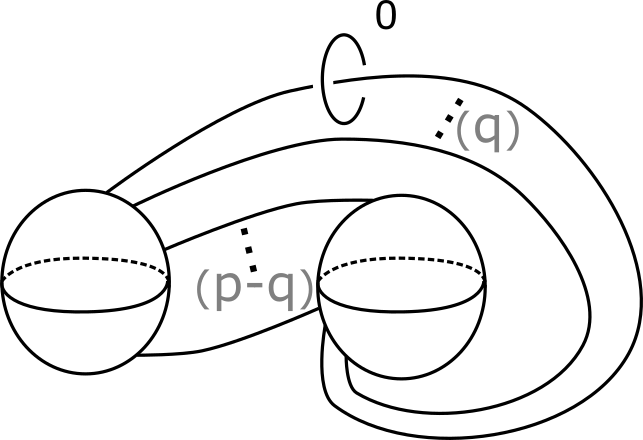}
        \caption{A \Kd of spun lens space.}
    \end{subfigure}
    \hfill
    \begin{subfigure}{0.3\textwidth}
        \centering
        \includegraphics[width=\textwidth]{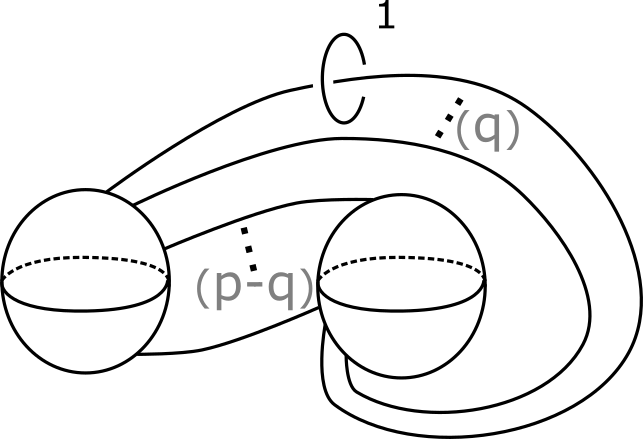}
        \caption{A Kirby diagram of twisted spun lens space.}
    \end{subfigure}
    \caption{(For clarity, blackboard framing is omitted here.)}
    \label{fig:generalized lens space}
    \end{figure}

\subsection{Open books with page handlebody}

Let us look at \Kds of the family of open books with page handlebody $H_{g,0}$, $g\geq0$. Starting with the most basic case $g=0$ and trivial monodromy, we would like to construct a \Kd of $\operatorname{Ob}(D^3,\operatorname{id})$ using two different input \Hds of $D^3$. Firstly, the empty diagram is by definition the result of attaching nothing to $D^3$, thus is a \Hd of $D^3$. The empty diagram remains empty after applying each step in Algorithm~\ref{algorithm2}, and thus we obtain a \Kd of $S^4$. Secondly, by handle cancellation~\cite[Proposition~4.2.9]{gs}, the "dumbbell" as illustrated in Figure~\ref{fig:dumbbell} is another \Hd of $D^3$. By handle cancellation the output of Algorithm~\ref{algorithm2} is again a \Kd of $S^4$. Therefore, independent of the choice of input, we conclude that $\operatorname{Ob}(D^3,\operatorname{id})$ is diffeomorphic to $S^4$.
\begin{figure}[ht]
    
        \centering
        \includegraphics[scale=0.5]{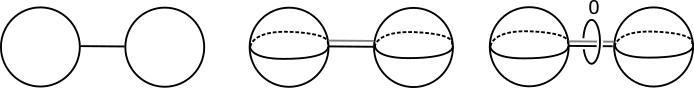}
     
    \caption{Left to right: a \Hd of $D^3$, a \Kd of half open book with page $D^3$, a \Kd of $S^4$.}
    \label{fig:dumbbell}
\end{figure}

Given $g>0$, each one of the $g$ \handles{1} of the page $H_{g,0}$ induces a \handle{1} of $\operatorname{hob}(H_{g,0})$. Thus, $\operatorname{hob}(H_{g,0})$ is diffeomorphic to the result $\#_g S^1\times D^3$ of attaching $g$ \handles{1} to $D^4$. A \Kd of this manifold with non-empty boundary, as given by Algorithm~\ref{algorithm1}, consists of $g$ pairs of $D^3$'s. Since an open book with trivial monodromy is obtained by taking the double of the half open book, $\operatorname{Ob}(H_{g,0},\operatorname{id})$ is diffeomorphic $\#_g S^1\times S^3$. The output \Kd of 
Algorithm~\ref{algorithm2}, also consisting of $g$ pairs of $D^3$'s, should be interpreted as the closed manifold $\#_g S^1\times S^3$.

Algorithms~\ref{algorithm2} and~\ref{algorithm3} are done after the first step in the absence of \handles{2}. Consequently, the output \Kd of $\operatorname{Ob}(H_{g,0},\operatorname{id})$ coincides with the output \Kd of $\operatorname{Ob}(H_{g,0},\varphi)$ for any monodromy $\varphi\colon H_{g,0}\rightarrow H_{g,0}$. This reflects the triviality of the mapping class group $\operatorname{MCG}(H_{g,0},\partial H_{g,0})$ -- \ two diffeomorphisms $F, F'$ on $H_{g,0}$ are isotopic if and only if their restrictions $\partial F, \partial F'$ to the boundary are isotopic~\cite{hensel}.

\section{Open books with page punctured solid torus}\label{monodromy on punctured handlebodies section}
The goal of this section is to introduce notions and techniques in preparation for Section~\ref{application}. Firstly, we define so-called \textit{torus twists} and \textit{sphere twists} monodromies on a punctured handlebody. Then we construct a \Kd of an open book with page punctured solid torus and monodromy a composition of torus and sphere twists. It turns out that such twists provide us with sufficiently many monodromies needed in the proof of Theorem~\ref{not unique}. Finally, we show how to obtain an entire family of \Kds of an open book with page punctured solid torus and monodromy a composition of torus and sphere twists.

\subsection{Twists along sphere and torus}
Let $S$ be an embedded surface with trivial normal bundle in a $3$-manifold $M$. Identifying a regular neighborhood of $S$ in $M$ with $S \times [0,1]$ and choosing a representation $f_t$ of a generator of $\pi_1(\operatorname{Diff}(S),\operatorname{id})$ gives rise to a monodromy on $M$ called \textit{twist along} $S$ given by
\begin{equation*}
          \begin{cases}
            f_t\times \operatorname{id_{[0,1]}} & \text{on } S\times [0,1], \\
            \operatorname{id} & \text{elsewhere.}
          \end{cases}
        \end{equation*}
Note that properly embedded spheres and tori are the only closed, orientable surfaces for which the group $\pi_1(\operatorname{Diff}(S),\operatorname{id})$ is non-trivial~\cite{EE}. In particular, the generator is unique for $S^2$ as seen in Section~\ref{hopf link section}.

Given a punctured handlebody $H_{g,n}$, we indicate the index of a handle in the superscript and enumerate handles of each index in the subscript as follows $$H_{g,n}=D^3\cup h^1_1\cup h^1_2\cup\dots\cup h^1_g\cup h^2_1\cup h^2_2\cup\dots\cup h^2_n.$$ 

\begin{definition}\label{st}
    Let $n>0$ and let $S^j$ be a properly embedded sphere in $H_{g,n}$ parallel to the boundary of $D^3\cup h^2_j\subset H_{g,n}$ for each $j=1,\dots,n$. A twist along $S^j$ is called a \textit{sphere twist} and is denoted by $\sigma^j$.
\end{definition}

A sphere twist $\sigma^j$ changes the framing of the cocore of the $j$-th \handle{2} while fixing the cocore. When $n=1$, we simply denote a sphere twist on the handlebody with one puncture by $\sigma$.
\begin{figure}[ht]
    \centering
    \includegraphics[scale=0.17]{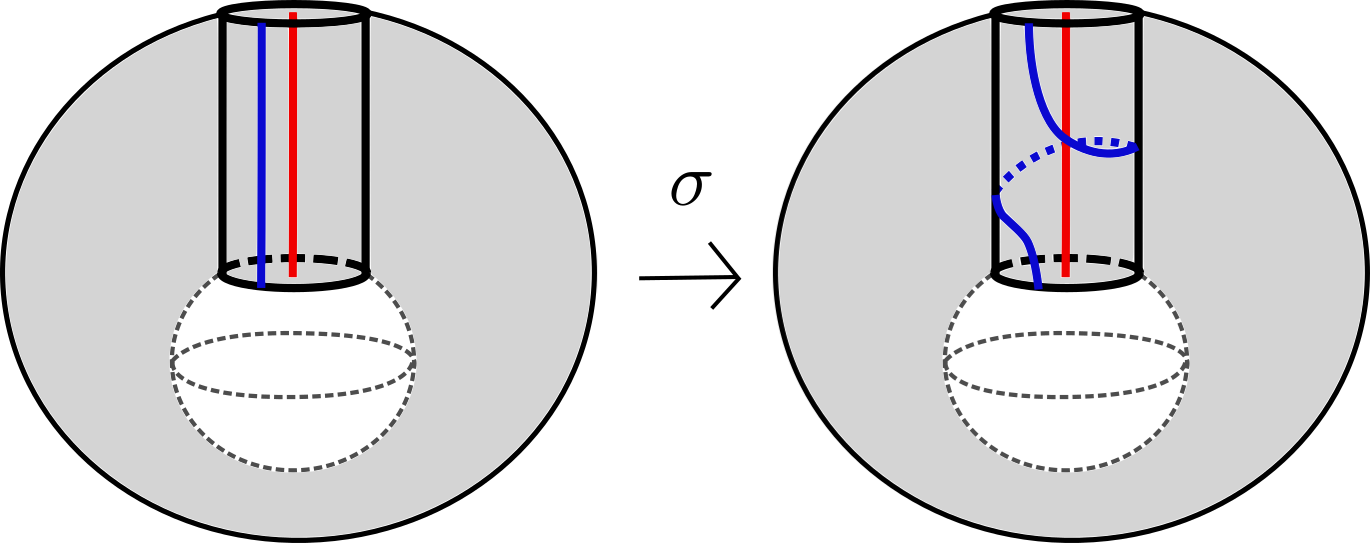}
    \caption{A \textit{sphere twist} on the punctured ball $H_{0,1}$.}
    \label{fig:sphere twist}
\end{figure}

\begin{definition}\label{tt}
    Let $g,n>0$ and let $T^j_l$ be a properly embedded torus in $H_{g,n}$ parallel to the boundary of $D^3\cup h^1_l\cup h^2_j\subset H_{g,n}$ for each $l=1,\dots,g$ and $j=1,\dots,n$. A twist along $T^j_l$ is called a \textit{torus twist} and is denoted by $\tau^j_l$.
\end{definition}

A torus twist $\tau^j_l$ pushes the puncture inside $D^3\cup H^2_j\subset H_{g,n}$ around the $l$-th \handle{1} and back to its starting position, or equivalently, it sends the cocore of the $j$-th \handle{2} to an arc in $H_{g,n}$ running through the $l$-th \handle{1}. Similarly, let $\tau$ denote a torus twist when $g=n=1$. 
\begin{figure}[ht]
        \includegraphics[scale=0.2]{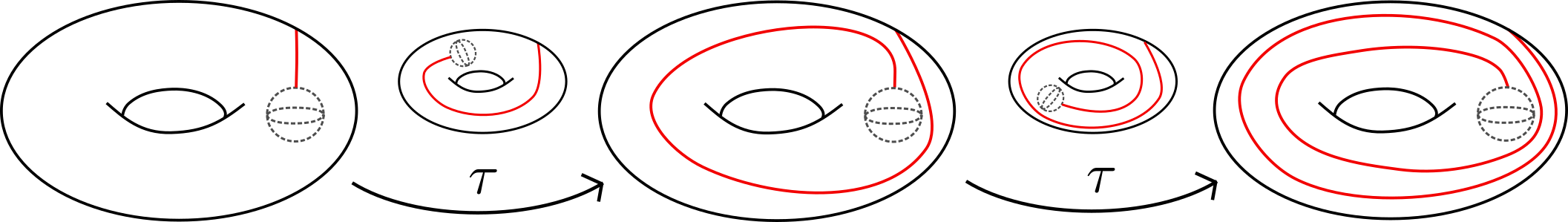}
        
        \caption{\textit{Torus twist} on the punctured solid torus $H_{1,1}$.}
        \label{fig:torus twist}
\end{figure}

Composing two torus twists pushes the puncture around the \handle{1} twice, thus the cocore of the \handle{2} gets mapped to an arc running through the \handle{1} twice and so on.

\begin{example}
The composition of torus twists $\tau^1_2\circ\tau^2_2\colon H_{2,2}\rightarrow H_{2,2}$ sends the cocore of the second \handle{2} to an arc running through the second \handle{1} and sends the cocore of the first \handle{2} to an arc also running through the second \handle{1} as demonstrated by Figure~\ref{fig:subtorus twist composition}.
\begin{figure}[ht]
    \centering
    \includegraphics[scale=0.2]{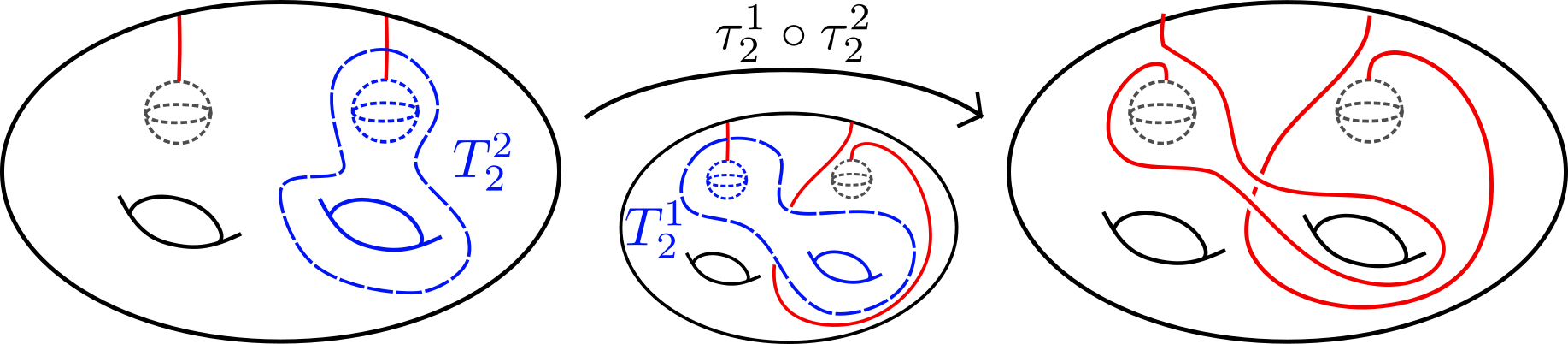}
    \caption{A twist along $T^2_2$ followed by a twist along $T^1_2$.}
    \label{fig:subtorus twist composition}
\end{figure}
\end{example}

\subsection{An example}\label{punctured torus}

    We apply Algorithm~\ref{algorithm3} step-by-step with inputs:\begin{itemize}
        \item a \Hd of the punctured solid torus $H_{1,1}$ as in Figure~\ref{fig:hd of punctured solid torus} and
        \item image under $\sigma\circ\tau\circ\tau\circ\tau$ of the \handle{2}. (Figure~\ref{fig:cocore under monodromy} only shows the image of the cocore of the \handle{2}.)
    \end{itemize}

\begin{figure}[ht]
        \centering
        \includegraphics[scale=0.2]
{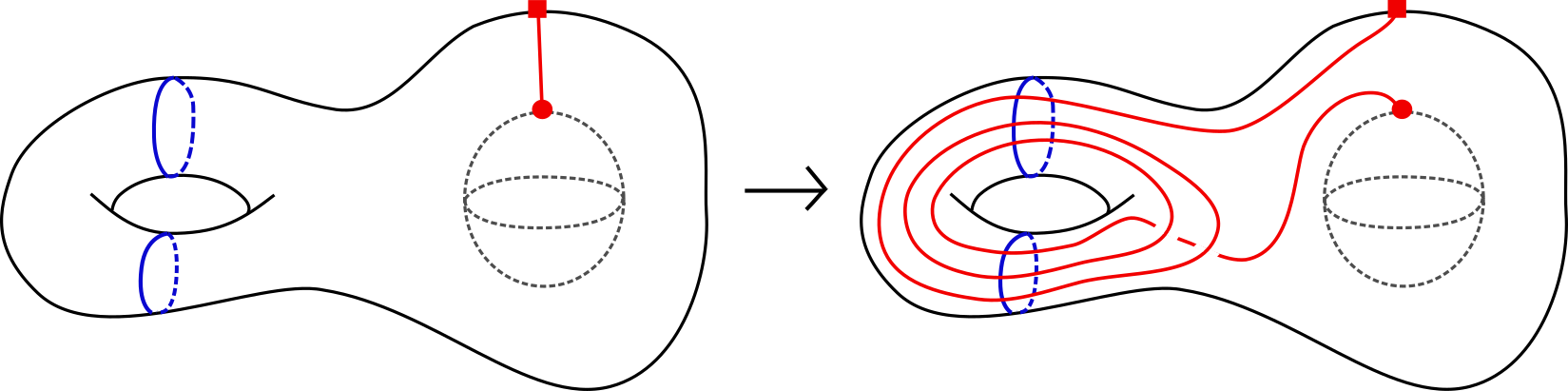}
       \caption{}
        \label{fig:cocore under monodromy}
\end{figure}    

Step~\ref{algo3 step1}: Use Algorithm~\ref{algorithm1} to get a \Kd of the half open book with page $H_{1,1}$, which consists of a pair of $D^3$'s and an unknot with blackboard framing in the $yz$-plane in $\mathbb{R}^3$. Step~\ref{add half-meridian}: Add behind the $yz$-plane a half-meridian with blackboard framing to the unknot. Step~\ref{algo3 monodromy}: Add in front of the $yz$-plane a half-meridian with blackboard framing to the unknot and replace it with its image under $\sigma\circ\tau\circ\tau\circ\tau$. In Figure~\ref{fig:cocore under monodromy}, the pair of circles (blue) on the torus each bound a disk, which then bounds a solid cylinder inside $H_{1,1}$. This solid cylinder corresponds to the \handle{1} as seen in Figure~\ref{fig:hd of punctured solid torus}. Observe that the cocore of the \handle{2} (red) gets mapped to an arc running through the \handle{1} three times. Thus the image of the half-meridian is an arc running through $D^3$ thrice. Since the monodromy contains a sphere twist, the image of a parallel arc of the cocore twists once along the image of the cocore.

\begin{figure}[ht]
    \centering
    \includegraphics[scale=0.27]{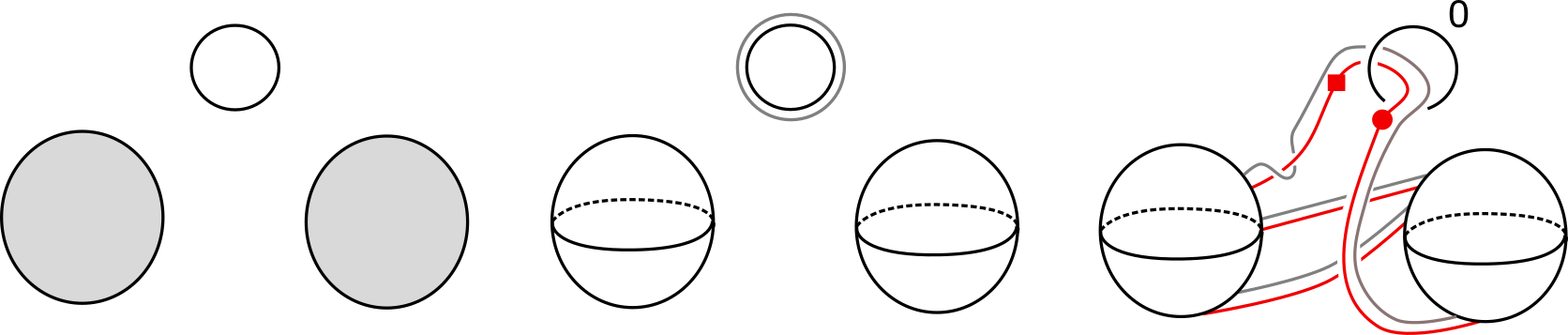}
    \caption{Left to right: a \Hd of $H_{1,1}$, a \Kd of half open book, a \Kd of $\operatorname{Ob}(H_{1,1},\sigma\circ\tau\circ\tau\circ\tau)$.}
    \label{fig:punctured torus}
\end{figure}

\subsection{Equivalence of Kirby diagrams}\label{braids}
When going from the middle to the right in Figure~\ref{fig:punctured torus}, one could have also obtained any \Kd shown in Figure~\ref{fig:family of kd}. The fact that these \Kds represent diffeomorphic manifolds will be useful when it comes to proving Corollaries~\ref{spun lens space cor 1} and~\ref{spun lens space cor}.
\begin{figure}[ht]
    \centering
    \includegraphics[scale=0.27]{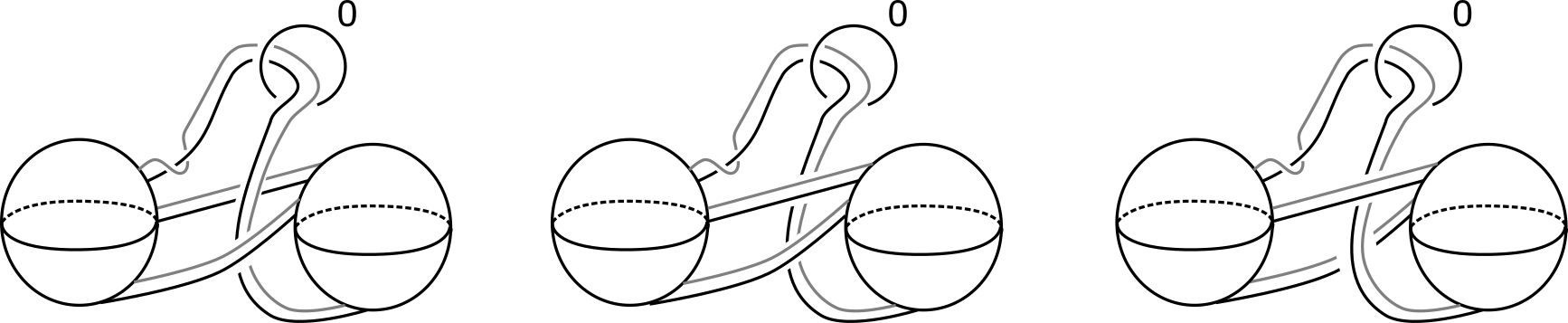}
    \caption{Some other \Kds of $\operatorname{Ob}(H_{1,1},\sigma\circ\tau\circ\tau\circ\tau)$.}
    \label{fig:family of kd}
\end{figure}

It suffices to show the family of diagrams in Figure~\ref{fig:family of kd'} are equivalent.
\begin{figure}[ht]
    \centering
    \includegraphics[scale=0.27]{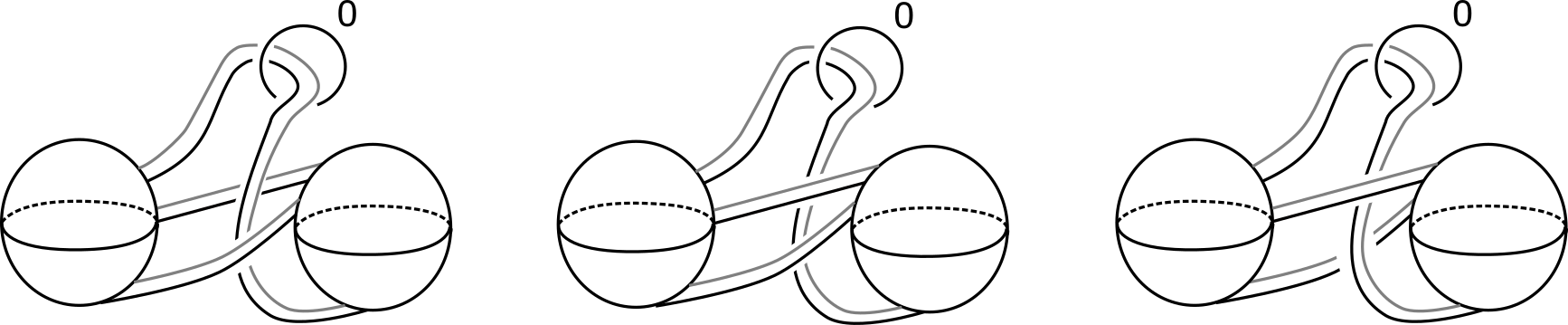}
    \caption{}
    \label{fig:family of kd'}
\end{figure}
\begin{definition}
    A \Kd is said to \textit{come from a $p$-braid} if it is a \Kd obtained from a $p$-stranded braid, or $p$-braid for short, by \begin{itemize}
    \item embedding the braid in the thickened $yz$-plane in $\mathbb{R}^3$,
    \item closing it off by attaching a pair of identified $D^3$'s,
    \item adding blackboard framing, and
    \item adding a $0$-framed meridian.
\end{itemize}
\end{definition}

\begin{figure}[ht]
        
            \centering
            
            \includegraphics[scale=0.3]{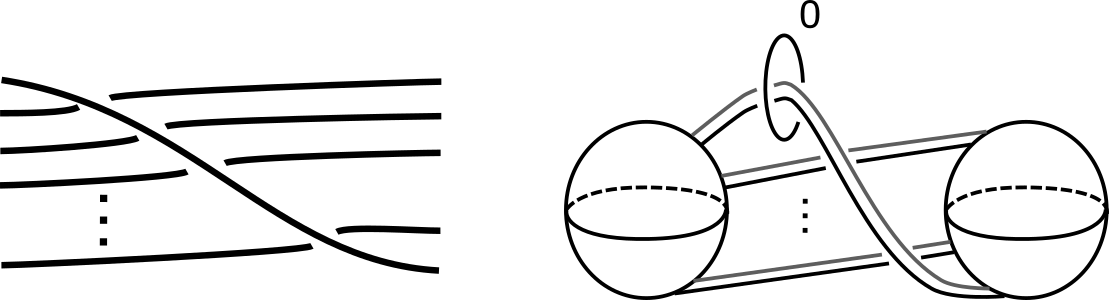}
        
        \caption{A \Kd coming from a braid whose closure is a knot.}
        \label{fig:braid}
    \end{figure}

We are interested in \Kds that come from a $p$-braid whose closure is a knot. Lemma~\ref{equivalence of kd} implies that the \Kds in Figure~\ref{fig:family of kd'} represent diffeomorphic $4$-manifolds.

\begin{lemma}\label{equivalence of kd}
    All \Kds that come from a $p$-braid whose closure is a knot represent diffeomorphic $4$-manifolds.
\end{lemma}

\begin{proof}
A $p$-braid induces a permutation on $p$ letters. First claim: \textit{$p$-braids that induce the same permutation on $p$ letters give rise to \Kds representing diffeomorphic $4$-manifolds.} If two non-equivalent $p$-braids induce the same permutation on $p$ letters, then they are related by crossing changes. By isotopy, move the $0$-framed meridian close to a crossing, then slide the \handle{2} over it to change the crossing as shown in Figure~\ref{fig:handle slide}. Note that the framing of both \handles{2} remains unchanged. The handle slide is a \Kd invariant move, hence proving the first claim. Consequently, a permutation on $p$ letters defines a Kirby diagram.

\begin{figure}[ht]
    \centering
    \includegraphics[scale=0.4]{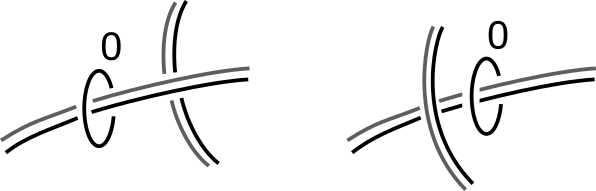}
    \caption{A handle slide changes the crossing.}
    \label{fig:handle slide}
\end{figure}

Second claim: \textit{permutations up to conjugation by a cycle of length two give rise to \Kds representing diffeomorphic $4$-manifolds.} Isotope the \handle{2} that intersects with the attaching region $D^3$ along $\partial D^3$. We can switch the places of two intersections without changing the framing of the \handle{2} as shown in Figure~\ref{fig:handle perturbation}. This \Kd isotopy corresponds to adding a conjugate pair of braids to each end of a braid as shown in Figure~\ref{fig:braid twist}, or equivalently, to conjugating the induced permutation by a cycle of length two. This proves the second claim.

\begin{figure}[ht]
\begin{subfigure}{0.55\textwidth}
    \centering
    \includegraphics[scale=0.28]{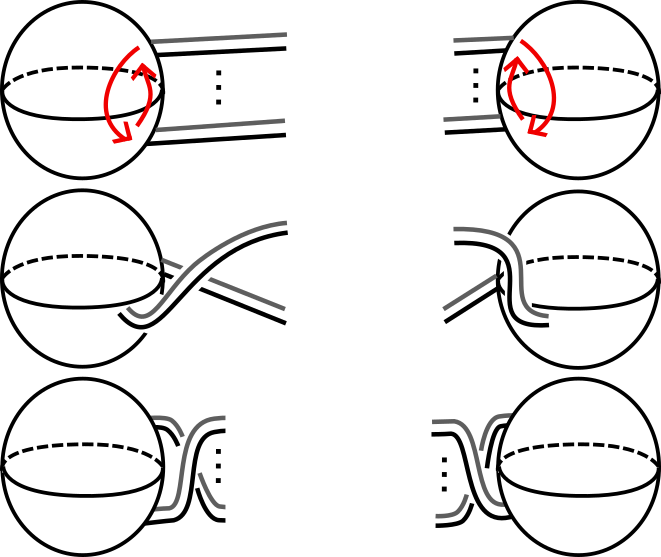}
    \caption{Isotope the \handle{2} along the boundary of the $D^3$ switching the places of two intersections.}
    \label{fig:handle perturbation}
\end{subfigure}\hfill
\begin{subfigure}{0.35\textwidth}
    \centering
    \includegraphics[scale=0.33]{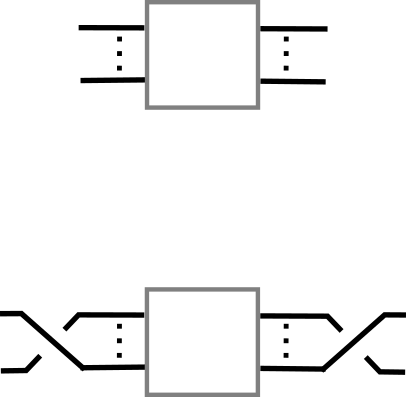}
    \caption{This corresponds to adding a conjugate pair of braids.}
    \label{fig:braid twist}
\end{subfigure}
\caption{}
\end{figure}

On one hand, if two $p$-braids induce the same permutation, then we are done by the first claim. On the other hand, if two $p$-braids induce different permutations, then $p>1$. A $p$-braid whose closure is a knot induces a permutation with no fixed points that can be represented by a cycle $(a_1a_2\dots a_p)$ of length $p$. According to the following computation, for $i\neq j$,
$$(a_ia_j)(a_1\dots a_i\dots a_j\dots a_p)(a_ja_i)=(a_1\dots a_j\dots a_i\dots a_p),$$
permutations on $p$ letters with no fixed points are all equivalent up to conjugations by a cycle of length two. The proof is complete by the second claim.
\end{proof}

\section{Applications}\label{application}
One can perturb a \Kd of a $4$-dimensional open book with trivial monodromy in such a way that one recognizes a different open book decomposition of the same $4$-manifold as demonstrated in Figure~\ref{fig:not unique}. For example, it follows that $\operatorname{Ob}(\operatorname{L}(3,1)-D^3,\operatorname{id})$ is diffeomorphic to $\operatorname{Ob}(H_{1,1},\sigma\circ\tau\circ\tau\circ\tau)$. With this observation and Proposition~\ref{monodromy on punctured handlebodies}, we prove that any open book constructed with trivial monodromy admits an open book decomposition with a punctured handlebody as page. Corollary~\ref{spun lens space cor} shows that the spin of a lens space $\operatorname{L}(p,q)$ only depends on $p$~\cite{Meier,Pao,Plotnick}.

\begin{figure}[ht]
    \centering
    \includegraphics[scale=0.265]{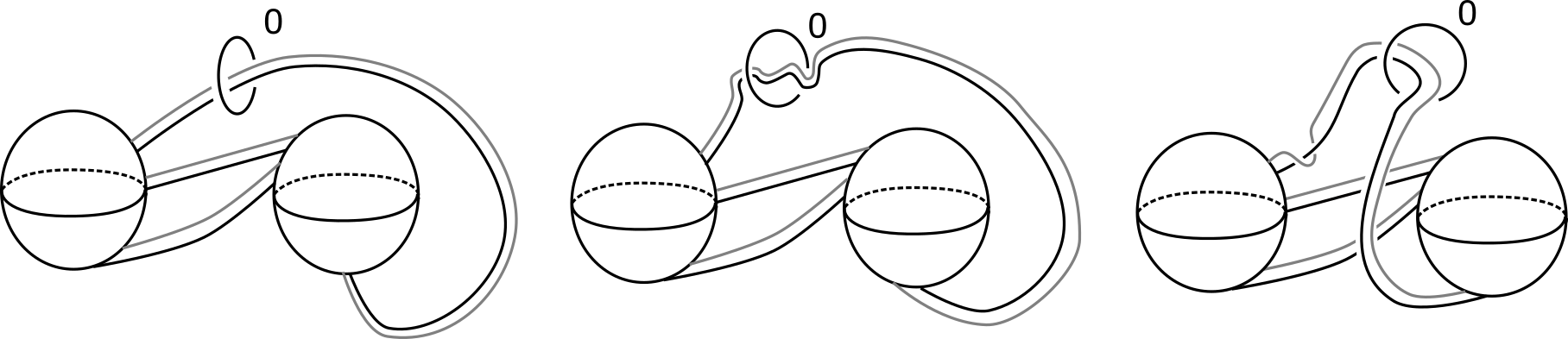}
    \caption{Left to right: Perturb the \Kd by rotating the $0$-framed meridian from a position transverse to the paper to lay down on the paper. We have encountered the \Kd on the left in Figure~\ref{fig:lens space}, it is a \Kd of $\operatorname{Ob}(\operatorname{L}(3,1)-D^3,\operatorname{id})$. The \Kd on the right represents $\operatorname{Ob}(H_{1,1},\sigma\circ\tau\circ\tau\circ\tau)$ as seen in Figure~\ref{fig:punctured torus}.}
    \label{fig:not unique}
\end{figure}

The following proposition delivers the existence of a desired monodromy for Theorem~\ref{not unique}.
\begin{proposition}\label{monodromy on punctured handlebodies}
     Given $g,n> 0$. There exists a monodromy on $H_{g,n}$ that sends the cocore of the $j$-th \handle{2} to an arc in $H_{g,n}$ running through the \handles{1} $$h_{j_1},\dots,h_{j_{k(j)}},$$ where $\{j_1,\dots,j_{k(j)}\}\subset\{0,1,\dots,g\}$, for all $j=1,\dots,n$.
\end{proposition}
\begin{proof}
    It follows by Definition~\ref{tt} that $$\varphi=\mathcal{T}_n\circ\dots\circ\mathcal{T}_2\circ\mathcal{T}_1,\text{ where }\mathcal{T}_i=\tau^i_{i_{k(i)}}\circ\dots\circ\tau^i_{i_1}$$
    for all $i=1,\dots,n$, is a monodromy on $H_{g,n}$ that sends the cocore of the $j$-th \handle{2} to an arc in $H_{g,n}$ running through the \handles{1} $$h_{j_1},\dots,h_{j_{k(j)}}$$ for all $j=1,\dots,n$.
\end{proof}

\begin{theorem}\label{not unique}
    Given any compact oriented $3$-manifold $M$ with non-empty boundary, there exists a monodromy $\varphi$ on a punctured handlebody $H$ such that $\operatorname{Ob}(H,\varphi)$ is diffeomorphic to $\operatorname{Ob}(M,\operatorname{id})$. Furthermore, suppose $M$ admits a handle decomposition with a single \handle{0}, $g$ \handles{1} and $n$ \handles{2}, then $H=H_{g,n}$ and $\varphi$ is given by a composition of sphere twists and torus twists. Namely, a torus twist $\tau^j_l$ each time the attaching sphere of the $j$-th \handle{2} runs through the $l$-th \handle{1} in $M$.
\end{theorem}
\begin{proof}
    Suppose $M$ is not a punctured handlebody, otherwise, just take $H=M$ and the theorem is trivial. Given a \Hd of $M$ that came from a standard handle decomposition without \handles{3}, with $n$ \handles{2} (note that $n>0$, since $n=0$ would imply $M$ is a handlebody) and $g$ \handles{1}, Algorithm~\ref{algorithm1} produces a \Kd of the half open book with page $M$ that consists of $n$ blackboard-framed \handle{2} attaching spheres $\alpha_1,\dots,\alpha_n$ and $g$ pairs of $D^3$'s. According to Algorithm~\ref{algorithm2}, adding a $0$-framed meridian $\beta_i$ to each $\alpha_i$ gives a \Kd of the open book $\operatorname{Ob}(M,\operatorname{id})$. Note that $\alpha_i$ lies in the $yz$-plane and $\beta_i$ is transverse to the $yz$-plane for all $i=1,\dots,n$. We will perturb this \Kd of $\operatorname{Ob}(M,\operatorname{id})$ to a \Kd of $\operatorname{Ob}(H_{g,n},\varphi)$, for some monodromy $\varphi\colon H_{g,n}\rightarrow H_{g,n}$ given by a composition of twists along spheres and tori embedded in $H_{g,n}$.
\begin{figure}[ht]
        \begin{subfigure}{0.45\textwidth}
            \centering
            \includegraphics[scale=0.45]{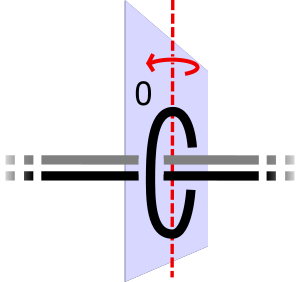}
            \caption{Part of a framed \handle{2} attaching sphere $\alpha_i$ and its $0$-framed meridian $\beta_i$. $\alpha_i$ is in the $yz$-plane, while $\beta_i$ lies on a plane (blue) transverse to the $yz$-plane.}
        \end{subfigure}
        \hfill
        \begin{subfigure}{0.45\textwidth}
            \centering
            \includegraphics[scale=0.4]{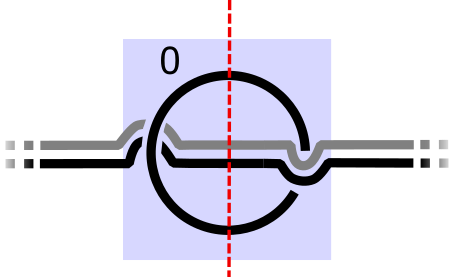}
            \caption{Rotate about the dashed line (red) so that the $0$-framed meridian lies on the $yz$-plane. This forces two small intervals of $\alpha_i$ to leave the $yz$-plane.}
            \label{fig:immersion}
        \end{subfigure}
        \caption{}
        \label{fig:perturb kd}
    \end{figure}
    
    Perturb this \Kd of $\operatorname{Ob}(M,\operatorname{id})$ so that each meridian $\beta_i$ becomes an unknot in the $yz$-plane. As a result, two small intervals of the \handle{2} attaching sphere $\alpha_i$, $i=1,\dots,n$, are forced to leave the $yz$-plane as shown in Figure~\ref{fig:perturb kd}. The \textit{perturbed \Kd} of $\operatorname{Ob}(M,\operatorname{id})$ still represents the same $4$-manifold. We would like to consider the perturbed \Kd as the union of $\alpha_i$, $i=1,\dots,n$, and a \Kd of some half open book $X$. We remove all the $\alpha_i$'s from the perturbed \Kd and refer to the resulting \Kd as $X$. To see the page of the half open book $X$ we perform Algorithm~\ref{algorithm1} in a reversed manner: remove the framing coefficients of the \handle{2} attaching spheres in $X$ and replace all $D^3$'s with $D^2$'s. This leaves us with $n$ unknots and $g$ pairs of $D^2$'s in the $yz$-plane, which is a \Hd of the punctured handlebody $H_{g,n}$. Hence the perturbed \Kd of $\operatorname{Ob}(M,\operatorname{id})$ is the union of the framed \handle{2} attaching spheres $\alpha_1,\dots,\alpha_n$ and a \Kd of the half open book with page $H_{g,n}$. It remains to show the existence of a monodromy $\varphi\colon H_{g,n}\rightarrow H_{g,n}$ such that the \Kd of $\operatorname{Ob}(H_{g,n},\operatorname{\varphi})$ is given by adding $\alpha_i$, $i=1,\dots,n$, to $X$.
    
    Let us see what properties the desired $\varphi$ must satisfy by analyzing our target, the perturbed Kirby diagram of $\operatorname{Ob}(M,\operatorname{id})$. \begin{itemize}
        \item The framed arc that is immersed behind the $yz$-plane would be the image of the attaching sphere in the back cover of the half open book with page $H_{g,n}$. (Recall Step~\ref{add half-meridian} of Algorithm~\ref{algorithm3}: add behind the $yz$-plane a half-meridian with blackboard framing.)
        \item The framed arc that popped out of the $yz$-plane together with the rest of that framed \handle{2} attaching sphere in the $yz$-plane would be the image of the attaching sphere in the front cover of the half open book. (Recall the part of a \Kd in front of the $yz$-plane is identified with the front cover of the half open book with page $H_{g,n}$.)
    \end{itemize}
    Suppose $\alpha_j$ runs through the \handles{1} $h_{j_1},h_{j_2},\dots,h_{j_{k(j)}}$, possibly with multiplicity, then the desired monodromy $\varphi$ must send the cocore of the $j$-th \handle{2} to an arc in $H_{g,n}$ running through the \handles{1} $h_{j_1},h_{j_2},\dots,h_{j_{k(j)}}$ for all $j=1,\dots,n$. There exists such a monodromy by Proposition~\ref{monodromy on punctured handlebodies}, namely $$\varphi=\mathcal{T}_n\circ\dots\circ\mathcal{T}_2\circ\mathcal{T}_1,\text{ where }\mathcal{T}_j=\tau^j_{h_{j_{k(j)}}}\circ\dots\circ\tau^j_{h_{j_1}}$$ for all $j=1,\dots,n$. (If $\alpha_l$ does not run through any $D^3$'s in the \Kd of $\operatorname{Ob}(M,\operatorname{id})$, then $\mathcal{T}_l$ is just the identity map. This cannot be the case for all $l\in\{1,\dots,n\}$ because $M$ is by assumption, not a punctured handlebody.) Compose $\varphi$ with a composition of sphere twists $\sigma^j$ to correct the framing of the image of $\alpha_j$, $j=1,\dots,n$, if necessary.
\end{proof}

\begin{corollary}\label{spun lens space cor 1}
    Every spun lens space is an open book with the punctured solid torus $H_{1,1}$ as page. More specifically, $\operatorname{Ob}(\operatorname{L}(p,q)-D^3,\operatorname{id})$ is diffeomorphic to $\operatorname{Ob}(H_{1,1},(\sigma)^q\circ(\tau)^p)$, where $(\sigma)^q\circ(\tau)^p$ denotes a composition of $p$ torus twists and $q$ sphere twists.
\end{corollary}
\begin{proof}

It follows Algorithm~\ref{algorithm3} that a \Kd coming from a $p$-braid, whose closure is a knot, is a \Kd of $\operatorname{Ob}(H_{1,1},(\tau)^p)$. By Lemma~\ref{equivalence of kd} we may consider the $p$-braid given in Figure~\ref{fig:part c}. One can check that the braid in Figure~\ref{fig:part c} is equivalent to the product of $q$ copies of the braid in Figure~\ref{fig:part a}.

\begin{figure}[ht]
\begin{subfigure}{0.3\textwidth}
    \centering
    \includegraphics[scale=0.33]{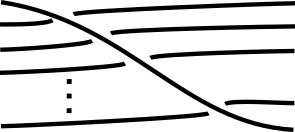}
    \caption{$q=1$}
    \label{fig:part a}

    \centering
    \includegraphics[scale=0.33]{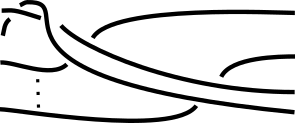}    \caption{$q=2$}

    \centering

\end{subfigure}
\hfill
\begin{subfigure}{0.65\textwidth}
    \centering
    \includegraphics[scale=0.5]{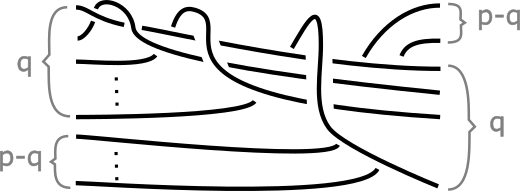}
\caption{}
\label{fig:part c}
\end{subfigure}
\caption{A family of $p$-braids whose closure are knots.}
    \label{fig:irreducible p-braid}
\end{figure}

To get a \Kd of $\operatorname{Ob}(H_{1,1},(\sigma)^q\circ(\tau)^p)$, we simply add $q$ twists to the parallel knot as shown in Figure~\ref{fig:exotic p-braid}.
\begin{figure}[ht]
    \centering
    \includegraphics[scale=0.3]{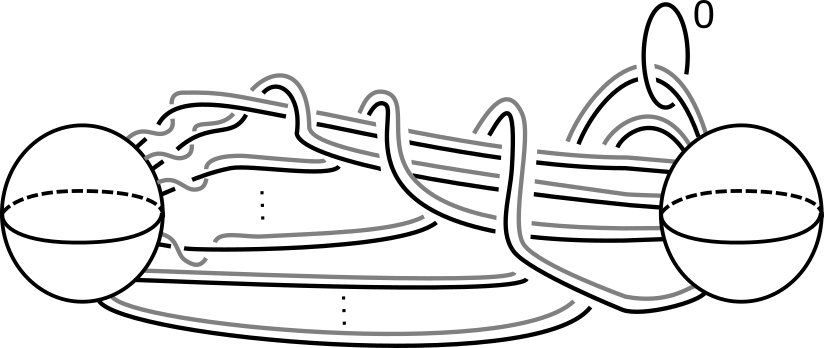}
    \caption{A \Kd of $\operatorname{Ob}(H_{1,1},(\sigma)^q\circ(\tau)^p)$.}
    \label{fig:exotic p-braid}
\end{figure}
The braid has been chosen so that the \Kd in Figure~\ref{fig:exotic p-braid} can be isotoped to a \Kd of the spun lens space $\operatorname{Ob}(\operatorname{L}(p,q)-D^3,\operatorname{id})$ as in Figure~\ref{fig:generalized lens space} by pulling the $q$ over-strands outwards. This shows that $\operatorname{Ob}(H_{1,1},(\sigma)^q\circ(\tau)^p)$ is diffeomorphic to $\operatorname{Ob}(\operatorname{L}(p,q)-D^3,\operatorname{id})$.

The case $(p,q)=(5,4)$ is demonstrated in Figure~\ref{fig:exotic p-braid example}.
\end{proof}

\begin{figure}[ht]
    \centering
    \includegraphics[scale=0.2]{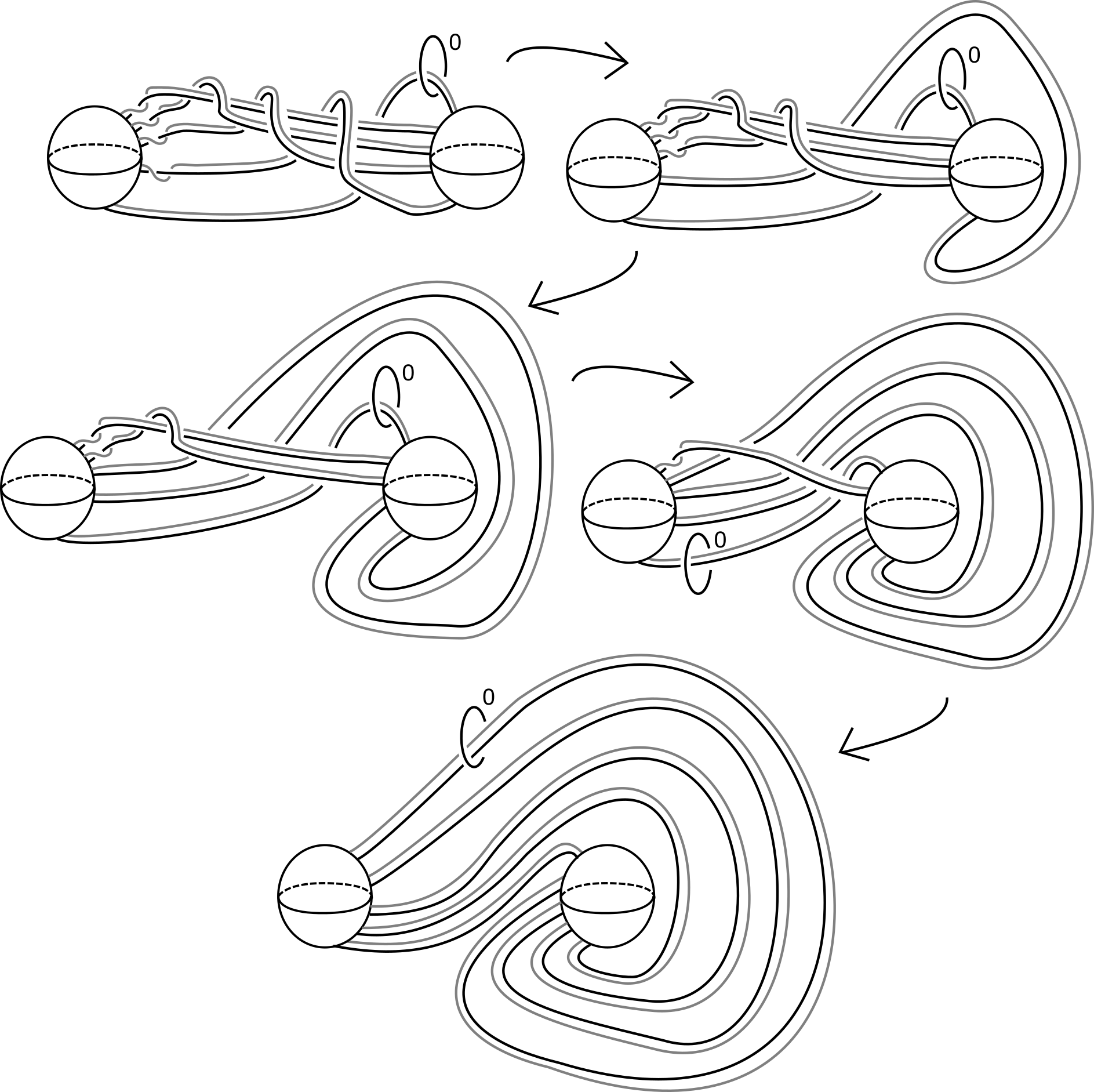}
    \caption{Showing that $\operatorname{Ob}(H_{1,1},(\sigma)^4\circ(\tau)^5)$ is diffeomorphic to $\operatorname{Ob}(\operatorname{L}(5,4)-D^3,\operatorname{id})$.}
    \label{fig:exotic p-braid example}
\end{figure}

Corollary~\ref{spun lens space cor} is a particular case of Pao's and Plotnick's work~\cite{Pao, Plotnick}, placed into context by Meier~\cite[Corollary~2.4]{Meier}. We provide a Kirby calculus proof.

\begin{corollary}\label{spun lens space cor}
    For all $1 \leq q < p$, the diffeomorphism type of a spun lens space $\operatorname{Ob}(\operatorname{L}(p,q)-D^3,\operatorname{id})$ is independent of $q$.
\end{corollary}
\begin{proof}
We first show that the diffeomorphism type of a spun lens space $\operatorname{Ob}(\operatorname{L}(p,q)-D^3,\operatorname{id})$ is independent of the parity of $q$, for $q<2p$.

We will show that $\operatorname{Ob}(\operatorname{L}(p,2k)-D^3,\operatorname{id})$, $k\in\mathbb{N}$, $p,2k$ coprime and $2k<2p$, are diffeomorphic to each other. Similarly, $\operatorname{Ob}(\operatorname{L}(p,2k+1)-D^3,\operatorname{id})$, $k\in\mathbb{N}$, $p, 2k+1$ coprime and $2k+1<2p$, are diffeomorphic to each other. The proof of the latter is essentially the same and hence omitted.
    
    When $p>2k$, consider the \Kd of $\operatorname{Ob}(\operatorname{L}(p,2k)-D^3,\operatorname{id})$ shown in Figure~\ref{fig:generalized lens space}. There are $2k$ strands to the right of the $D^3$ on the right-hand side, perform the isotopy shown in Figure~\ref{fig:pulling b} to each one of them.
    \begin{figure}[ht]
    \begin{subfigure}{0.45\textwidth}
        \centering
        \includegraphics[scale=0.25]{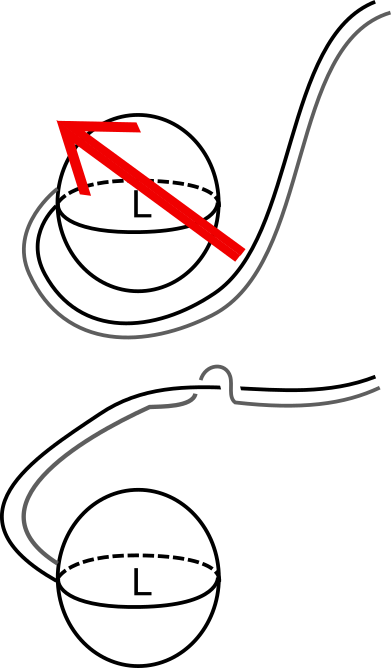}
        \caption{Pulling over the left $D^3$}
        \label{fig:pulling a}
    \end{subfigure}\hfill
    \begin{subfigure}{0.45\textwidth}
        \centering
        \includegraphics[scale=0.25]{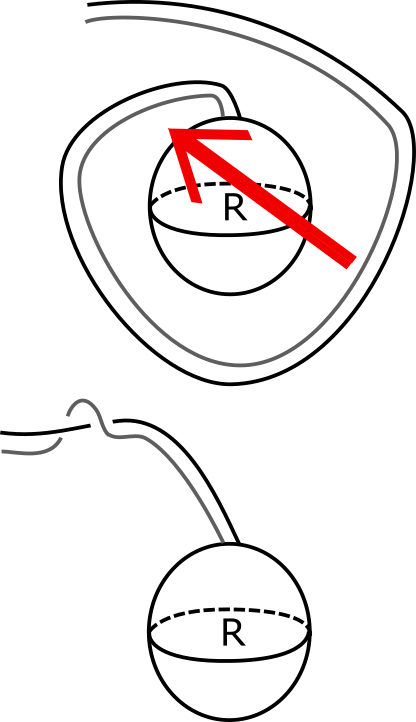}
        \caption{Pulling over the right $D^3$}
        \label{fig:pulling b}
    \end{subfigure}
    \caption{}
\end{figure} As a result, the parallel knot gains $2k$ twists in total. Slide the \handle{2} over the $0$-framed meridian $k$ times to undo these $2k$ twists restoring blackboard framing as shown in Figure~\ref{fig:handle slide type 2}. Hence we have shown that a \Kd of $\operatorname{Ob}(\operatorname{L}(p,2k)-D^3,\operatorname{id})$ and a \Kd that comes from a $p$-braid whose closure is a knot represent diffeomorphic $4$-manifolds for each integer $2k <p$ coprime with $p$.

\begin{figure}[ht]
        \centering
        \includegraphics[scale=0.44]{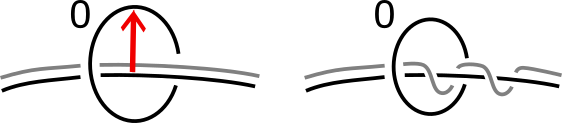}
        \caption{Slide the blackboard-framed \handle{2} over the $0$-framed meridian to change the framing by two.}
        \label{fig:handle slide type 2}
    \end{figure}
    
When $p<2k<2p$, construct a \Kd of $\operatorname{Ob}(\operatorname{L}(p,2k)-D^3,\operatorname{id})$ with Algorithm~\ref{algorithm2} using the \Hd in Figure~\ref{fig:lens hd extended} as input.
        \begin{figure}[ht]
        \centering
        \includegraphics[scale=0.27]{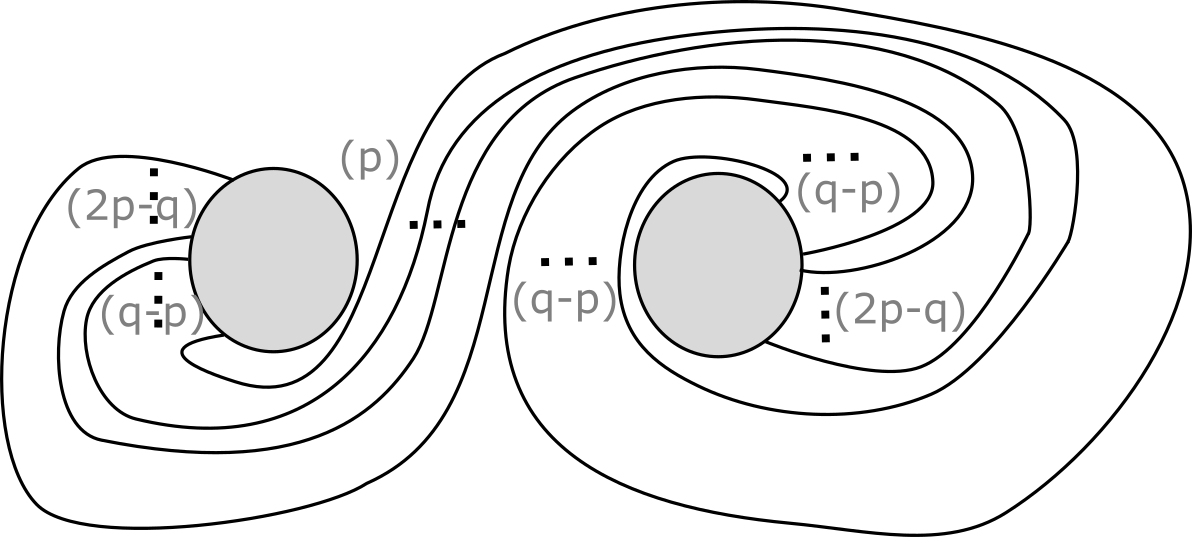}
        \caption{A \Hd of $\operatorname{L}(p,q)$, where $p<q<2p$ and $p,q$ coprime. The attaching sphere of the \handle{2} travels around the $D^2$ $q$ times and runs through the \handle{1} $p$ times.}
        \label{fig:lens hd extended}
    \end{figure}
    There are $2k$ strands between the centers of the $D^3$'s, we free this space as follows: counting from the left, pull the first $p$ strands (as in Figure~\ref{fig:pulling a}) over the left $D^3$ to the other side. There are also $2k$ strands to the right of the right $D^3$, counting from the right, pull the first $2k-p$ strands (as in Figure~\ref{fig:pulling b}) over the right $D^3$ to the other side.

    As a result, the parallel knot gains $2k$ twists in total. We slide the blackboard-framed \handle{2} over the $0$-framed meridian $k$ times to restore blackboard framing. Having cleaned up the space between the $D^3$'s, we rotate a $D^3$ (the other $D^3$ rotates synchronously) so that a $p$-braid is formed in between the pair of $D^3$'s. This shows that a \Kd of $\operatorname{Ob}(\operatorname{L}(p,2k)-D^3,\operatorname{id})$ and a \Kd that comes from a $p$-braid whose closure is a knot represent diffeomorphic $4$-manifolds for each integer $2k\in (p,2p)$ coprime with $p$.
    
    By Lemma~\ref{equivalence of kd}, \Kds that come from a $p$-braid whose closure is a knot represent diffeomorphic $4$-manifolds. Therefore $\operatorname{Ob}(\operatorname{L}(p,2k)-D^3,\operatorname{id})$, $k\in\mathbb{N}$, $p,2k$ coprime and $2k<2p$, are diffeomorphic to each other.
 \begin{figure}[ht]
        \centering
        \includegraphics[width=\textwidth]{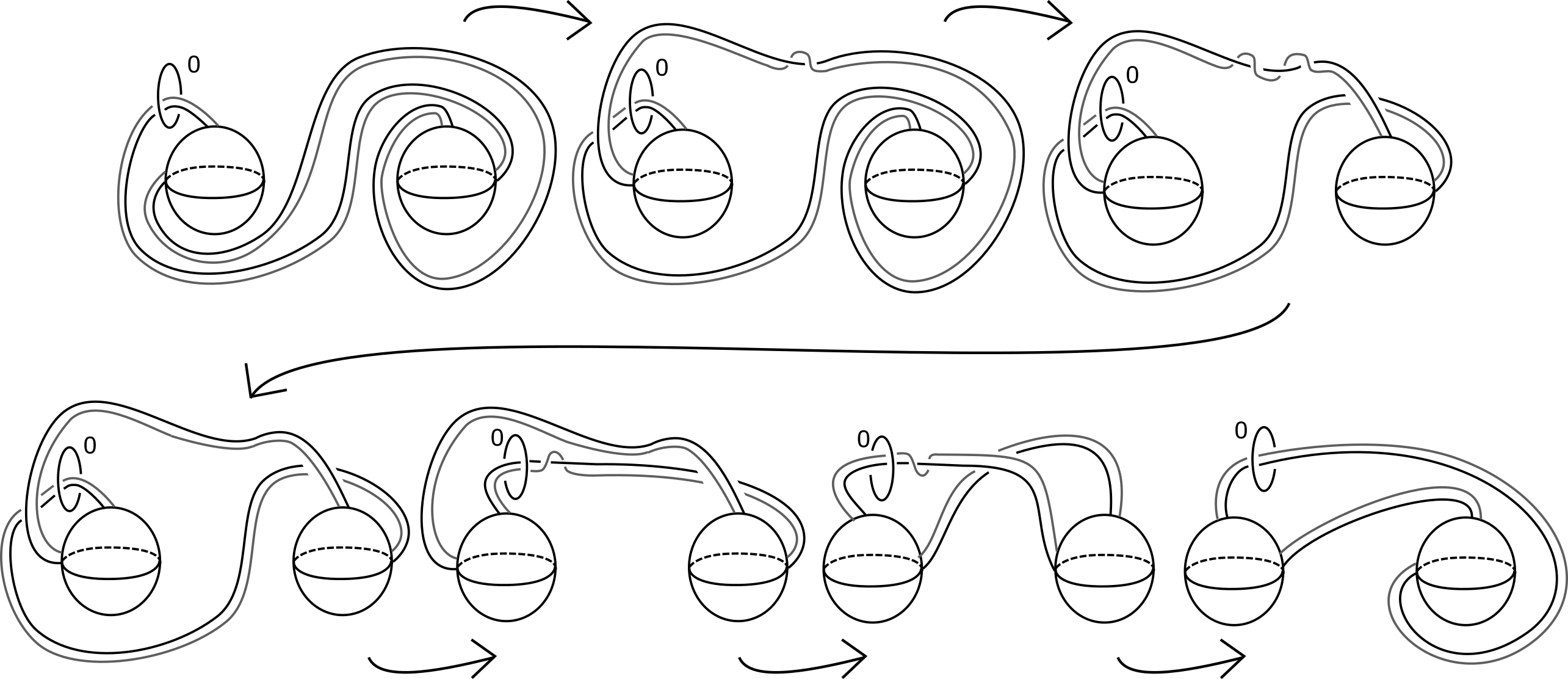}
        \caption{$\operatorname{Ob}(\operatorname{L}(2,3)-D^3,\operatorname{id})$ is diffeomorphic to $\operatorname{Ob}(\operatorname{L}(2,1)-D^3,\operatorname{id})$. To go from the first row to the second, slide the other \handle{2} over the $0$-framed meridian. In the middle of the second row, we rotate the pair of $D^3$.}
        \label{fig:lens extended example}
    \end{figure}
    
       Finally, by Rolfsen's twist $\operatorname{L}(p,q)$ is diffeomorphic to $\operatorname{L}(p,q+np)$ for all $n\in\mathbb{N}$~\cite{knots}, thus $\operatorname{L}(p,2k)$, $p>2k$, is diffeomorphic to $\operatorname{L}(p,2k+p)$. By assumption, $p$ and $2k$ are coprime, hence $2k+p$ is odd. Since we have already established that the diffeomorphism type is independent of the parity of $q$, for $q<2p$, this completes the proof.
 \end{proof}

\let\MRhref\undefined
\bibliographystyle{hamsalpha}  
\bibliography{Sources}  
\end{document}